\documentclass[11pt]{amsart}
\usepackage{amssymb,amsthm,amsmath,amstext,tikz-cd}
\usepackage{mathrsfs}  
\usepackage{bm}        
\usepackage{mathtools} 
\usepackage{changepage,enumitem}
\usepackage{color,fullpage}

\usepackage[all]{xy}
 \DeclareFontFamily{U}{wncy}{}
    \DeclareFontShape{U}{wncy}{m}{n}{<->wncyr10}{}
    \DeclareSymbolFont{mcy}{U}{wncy}{m}{n}
    \DeclareMathSymbol{\Sha}{\mathord}{mcy}{"58} 

\theoremstyle{plain}
\newtheorem{theorem}{Theorem}[section]
\newtheorem{prop}[theorem]{Proposition}
\newtheorem{lemma}[theorem]{Lemma}
\newtheorem{cor}[theorem]{Corollary}

\newtheorem*{theoremintro*}{Theorem}

\newtheorem*{propintro*}{Proposition}

\theoremstyle{definition}
\newtheorem{defin}[theorem]{Definition}
\newtheorem{example}[theorem]{Example}
\newtheorem{examples}[theorem]{Examples}
\newtheorem{remark}[theorem]{Remark}
\newtheorem{remarks}[theorem]{Remarks}
\newtheorem{question}[theorem]{Question}
\newtheorem*{ack}{Acknowledgments}

\newcommand{\lgp}{\text{LGP}}

\newcommand{\C}{\mathbb C}
\newcommand{\F}{\mathbb F}

\newcommand{\HH}{\mathbb{H}}
\newcommand{\Q}{\mathbb Q}

\DeclareMathOperator{\Br}{\mathrm{Br}}

\newcommand{\Char}{\text{char }}

\usepackage[backref=page]{hyperref}
\hypersetup{pdftitle={Universal quadratic forms}}
\hypersetup{pdfauthor={Connor Cassady}}
\hypersetup{colorlinks=true,linkcolor=red,anchorcolor=blue,citecolor=blue}

\title[Refined notions]{Two refined notions in quadratic form theory}
\date{\today}
\author[Cassady]{Connor Cassady}
\thanks{\textit{Mathematics Subject Classification} (2010): 11E04, 11E81, 14G12 (primary);  12J10, 13J10 (secondary).
\\ 
\textit{Key words and phrases}: quadratic forms, Witt index, local-global principles, field invariants, valued fields. \\
The author was supported in part by NSF grants DMS-1805439 and DMS-2102987.
}

\begin{document}

\maketitle
\vspace{-0.6cm}
\begin{abstract}
We use the Witt index to define and study a refined notion of the local-global principle for isotropy of quadratic forms over a field $k$ and to define and study refined versions of the $m$-invariant of $k$. We also explore connections between these refinements.
\end{abstract}

\section*{Introduction}
Given a regular quadratic form $q$ over a field $k$ of characteristic $\ne 2$, the Witt Decomposition Theorem \cite[Theorem~I.4.1]{lam} implies that there is a unique non-negative integer called the \textit{Witt index} of~$q$, denoted by $i_W(q)$, and an anisotropic quadratic form $q_{an}$ over $k$ such that
\[
    q \simeq i_W(q)\HH \perp q_{an}.
\]
Here $\simeq$ denotes isometry of quadratic forms, $\HH$ denotes the \textit{hyperbolic plane} $x^2 - y^2$, $\perp$ denotes the orthogonal sum of two quadratic forms, and $i_W(q)\HH$ denotes the orthogonal sum of $i_W(q)$ copies of~$\HH$. The quadratic form $q$ is isotropic if and only if $i_W(q) \geq 1$, so the Witt index serves as a measure of the isotropy of $q$. The goal of this article is to use the Witt index to define and study a refined notion of the local-global principle for isotropy of quadratic forms over~$k$ and to define and study refined versions of the $m$-invariant of $k$.

We will first focus on refining the local-global principle for isotropy of quadratic forms over a given field, phrased in terms of discrete valuations on the field. Let $k$ be a field of characteristic $\ne 2$ equipped with a non-empty set $V$ of non-trivial discrete valuations. We say that a quadratic form $q$ over $k$ satisfies the \textit{local-global principle for isotropy with respect to $V$} if $q$ being isotropic over the $v$-adic completion $k_v$ of $k$ for all $v \in V$ implies that $q$ is isotropic over $k$. This local-global principle has been well-studied, and in many instances counterexamples have been found (see, e.g., \cite{auel-suresh, cas, cps12, gup18, gup21, hhk15}). So in these cases, the Witt index of a quadratic form $q$ being at least one everywhere locally (i.e., over every $v$-adic completion) does not guarantee that $i_W(q) \geq 1$ globally (i.e., over $k$). One motivation, then, for introducing a refined notion of the local-global principle for isotropy is to determine if requiring a larger Witt index for~$q$ everywhere locally guarantees that $i_W(q) \geq 1$ globally, and more generally, for some $r, s$, if $i_W(q) \geq r$ everywhere locally implies that $i_W(q) \geq s$ globally. If this implication holds, we say that $q$ \textit{satisfies $\emph{LGP}(r,s)$ with respect to $V$} (see Definition \ref{lgp(r,s)}). We will see that in some instances this can be done (see, e.g., Proposition \ref{I^n neighbors and lgp}), but in other cases, counterexamples to this refined local-global principle continue to exist (see Theorem~\ref{ce to lgp}). These counterexamples can be thought of as ``worse" than counterexamples to the original local-global principle for isotropy since these counterexamples are ``more isotropic" everywhere locally while remaining anisotropic globally.

The next notion we will use the Witt index to refine is the $m$-invariant of a field $k$ of characteristic~$\ne 2$. Recall that a quadratic form $q$ over $k$ is \textit{universal} if for every $a \in k^{\times}$ there is some~$x$ such that $q(x) = a$, and the $m$-\textit{invariant} of $k$, denoted by $m(k)$, is defined to be the minimal dimension of an anisotropic universal quadratic form over $k$ \cite{m-inv}. By the First Representation Theorem \cite[Corollary~I.3.5]{lam}, a regular quadratic form $q$ over $k$ is universal if and only if $q \perp \langle -a \rangle$ is isotropic for all $a \in k^{\times}$. So the First Representation Theorem allows us to relate the Witt index to the study of universal quadratic forms and the $m$-invariant. We therefore use the Witt index to define generalized invariants $m_{i,j}(k)$ for integers $i, j \geq 1$ (see Definition \ref{refined m}) in a fashion similar to how we pass from the local-global principle for isotropy to $\lgp(r,s)$. We will show how the invariants $m_{i,j}(k)$ behave as we vary~$i$ and~$j$ (see, e.g., Proposition \ref{decreasing i in m_{i,j}} and Corollary \ref{exact values for m_{i,j} for large j}), relate~$m_{i,j}(k)$ to the $m$- and $u$-invariant of $k$ (see, e.g., Theorem \ref{upper and lower bounds}), and show that these invariants $m_{i,j}(k)$ provide more information than $m(k)$ and $u(k)$ (see Theorem \ref{refined m determine u} and Corollary \ref{m_{i,j} separates}).

\textbf{Main results and structure}. We begin by introducing notation and proving several preliminary results in Section \ref{notation}. 

In Section \ref{LGP} we introduce and study a refined local-global principle for isotropy (see Definition~\ref{lgp(r,s)}). We begin this study in Subsection \ref{counterexamples} by constructing numerous counterexamples to this refined local-global principle for isotropy over purely transcendental extensions of fields $\ell \in \mathscr{A}_i(2)$ for some $i \geq 0$ with respect to a small set of discrete valuations (see Sections \ref{notation} and \ref{LGP} for notation and terminology). For example, all $C_i$ fields satisfy property $\mathscr{A}_i(2)$.
\begin{theoremintro*}[\ref{ce to lgp}]
Let $\ell$ be a field of characteristic $\ne 2$. Assume that $\ell \in \mathscr{A}_i(2)$ for some $i \geq 0$ and $u(\ell) = 2^i$. For any integer $r \geq 1$ let $L_r = \ell(x_1, \ldots, x_r)$, and for $r \geq 2$ let $V_r$ be the set of non-trivial discrete valuations on $L_r$ that are trivial on $L_{r-1}$. Then for $r \geq 2$ and any integer~$n$ such that $0 \leq n < 2^{i+r-2}$, there exists a $\left(2^{i+r}-n\right)$-dimensional quadratic form over $L_r$ that violates $\emph{LGP}\left(2^{i+r-2} - n, 1\right)$ with respect to $V_r$.
\end{theoremintro*}

Inspired by the existence of these counterexamples, and by the notion of Pfister neighbors, in Subsection \ref{I^n-neighbors} we define the notion of an $I^n$-\textit{neighbor} (see Definition \ref{I^n neighbors}) and prove the following (see Sections \ref{notation} and \ref{LGP} for notation and terminology).
\begin{propintro*}[\ref{I^n neighbors and lgp}]
Let $k$ be a field of characteristic $\ne 2$ equipped with a non-empty set $V$ of non-trivial discrete valuations with respect to which the local-global principle for isometry holds. Let $q$ be an $I^n$-neighbor of complementary dimension $r$ over $k$ for some $n \geq 1$. Then
\[
    q \text{ satisfies } \emph{LGP}\left(\frac{\dim q + r - 2^n}{2} + 1, \frac{\dim q - r}{2}\right) \text{ with respect to $V$}.
\]
\end{propintro*}

To conclude Section \ref{LGP}, in Subsection \ref{all forms} we consider all quadratic forms over a field $k$ and prove
\begin{theoremintro*}[\ref{varying r,s together}]
    Let $k$ be a field of characteristic $\ne 2$ equipped with a non-empty set $V$ of non-trivial discrete valuations. If there are integers $r, s \geq 1$ such that all quadratic forms over~$k$ satisfy $\emph{LGP}(r, s)$ with respect to $V$, then for any integer $j$ (not necessarily positive) such that $r+j, s+j \geq 1$, all quadratic forms over $k$ satisfy $\emph{LGP}(r+j, s+j)$ with respect to $V$.
\end{theoremintro*}
Theorem \ref{varying r,s together} motivates the introduction of a new invariant, denoted by $l(k, V)$, which is defined to be the smallest $r \geq 1$ such that all quadratic forms over $k$ satisfy $\lgp(r,1)$ with respect to $V$ (see Definition~\ref{l-inv}). We give an upper bound for $l(k,V)$ in Proposition \ref{first upper bound on l}.

In Section \ref{refined m-inv} we use the Witt index to define and study refined versions of the $m$-invariant of a field $k$, denoted by $m_{i,j}(k)$ (see Definition \ref{refined m}). In Subsection \ref{refined m generalities} we prove general properties of these refined invariants, and the main result of this subsection gives bounds for these refined $m$-invariants in terms of the~$m$- and $u$-invariant (see Subsection \ref{refined m generalities} for notation and terminology).
\begin{theoremintro*}[\ref{upper and lower bounds}]
    Let $k$ be a field of characteristic $\ne 2$ with $u(k) < \infty$, and let $i,j \geq 1$ be any positive integers. Then
    \[
        \max\{1, m(k) + 2j - 1- i\} \leq m_{i,j}(k) \leq \max\{1, u(k) + 2j - 1 - i\}.
    \]
Moreover, if either $u(k) + 2j - 1- i \leq 1$ or $m(k) = u(k)$, then both inequalities above are equalities.
\end{theoremintro*}
We also show that if two fields have the same refined $m$-invariants, then they have the same $u$-invariant.
\begin{theoremintro*}[\ref{refined m determine u}] Let $k, k'$ be fields of characteristic $\ne 2$. If $m_{i,1}(k) = m_{i,1}(k')$ for all integers~$i \geq 1$, then $u(k) = u(k')$.
\end{theoremintro*}
In Subsection \ref{refined m of cdvf} we consider a complete discretely valued field $K$ with residue field $k$, and use particular refined $m$-invariants of $k$ to compute certain refined $m$-invariants of $K$.
\begin{theoremintro*}[Propositions \ref{m_{i,1} of cdvf for small i} and \ref{m_{i,1} of cdvf for large i}]
    Let $K$ be a complete discretely valued field with residue field $k$ of characteristic $\ne 2$ such that $u(k) < \infty$. Let $i$ be any integer such that $1 \leq i < u(K)$. Then
    \[
        m_{i,1}(K) = \begin{cases}
            \min_{1 \leq r \leq i} \{m_{r,1}(k) + m_{i-r+1,1}(k)\} &\text{ if $1 \leq i \leq u(k)$}, \\
        \min_{\lceil \frac{i}{2} \rceil \leq s \leq u(k)}\{m_{s,1}(k) + m_{i-s+1,1}(k), m_{i-u(k),1}(k)\} &\text{ if $u(k) < i < u(K)$.}
        \end{cases}
    \]
\end{theoremintro*}
To conclude Section \ref{refined m-inv}, in Subsection \ref{refined m separates} we show that the values of $u(k)$ and $m(k)$ for a field $k$ do not determine the values of the invariants $m_{i,j}(k)$.
\begin{propintro*}[Corollary \ref{m_{i,j} separates}]
There are fields $k_1, k_2$ of characteristic $\ne 2$ with $u(k_1) = u(k_2)$ and $m(k_1) = m(k_2)$, but $m_{i,1}(k_1) \ne m_{i,1}(k_2)$ for $i = 2,3$.
\end{propintro*}
Finally, in Section \ref{connections} we explore connections between these two refined notions, $\lgp(r,s)$ and $m_{i,j}(k)$. For example, we prove
\begin{theoremintro*}[\ref{going-down}]
   Let $k$ be a field of characteristic $\ne 2$, let $V$ be a non-empty set of non-trivial discrete valuations on $k$, let $i,j \geq 1$ be positive integers, and let $s$ be an integer such that $1 \leq s \leq i$. If all quadratic forms over $k$ of dimension $m_{i,j}(k) + s - 1$ satisfy $\emph{LGP}(r, j)$ with respect to $V$ for some integer $r \geq 1$, then so do all quadratic forms over $k$ of dimension $< m_{i,j}(k)$.
\end{theoremintro*}
\section{Notation and preliminaries}
\label{notation}
All of the fields considered will have characteristic different from 2, and all quadratic forms considered (occasionally referred to just as forms) will be nondegenerate (or \textit{regular}). Our notation and terminology follows \cite{lam}, and we assume familiarity with basic notions of quadratic form theory (see, e.g., \cite[Chapter I]{lam}).

Let $q$ be an $n$-dimensional quadratic form over a field $k$ of characteristic $\ne 2$. Because $\Char k \ne 2$, we can diagonalize $q$ over $k$ and write $q \simeq \langle a_1, \ldots, a_n \rangle$ with each $a_i \in k^{\times}$. 
We call the elements $a_1, \ldots, a_n$ the \textit{entries} of $q$. If $K/k$ is any field extension, $q_K$ will denote the quadratic form $q$ considered as a quadratic form over $K$. For any $a \in k^{\times}$, we let $a \cdot q$ denote the form $\langle a \rangle \otimes q$, and $-q$ will denote the form $(-1) \cdot q$. Two quadratic forms $q_1, q_2$ over $k$ are \textit{similar} if $q_1 \simeq a \cdot q_2$ for some $a \in k^{\times}$. Given quadratic forms $q, \varphi$ over~$k$, we say that $\varphi$ is a \textit{subform} of $q$ if there exists some quadratic form $\psi$ over $k$ such that $q \simeq \varphi \perp \psi$. We say that an even-dimensional quadratic form~$\varphi$ over $k$ is \textit{hyperbolic} if $\varphi \simeq r\HH$ for some positive integer $r$, where $r\HH$ denotes the orthogonal sum of $r$ copies of $\HH$. Hence the Witt index $i_W(q)$ of a quadratic form $q$ over $k$ is half of the maximal dimension of a hyperbolic subform of $q$. The next two preliminary results about the Witt index will be used throughout this paper.

\begin{lemma}
\label{Witt index of sum}
Let $q$ and $\varphi$ be regular quadratic forms over a field $k$ of characteristic $\ne 2$. Then
\[
	i_W(q \perp \varphi) \leq i_W(q) + \dim \varphi.
\]
\end{lemma}
\begin{proof}
See, e.g., \cite[Exercise I.1.16(2)]{lam}.
\end{proof}
\begin{lemma}
\label{subform of hyperbolic form}
Let $n \geq 1$ be any positive integer, and for any $0 < r \leq n$, let $q$ be an $(n+r)$-dimensional regular subform of the hyperbolic form $n\HH$ over a field $k$ of characteristic $\ne 2$. Then $i_W(q) \geq r$.
\end{lemma}
\begin{proof}
See, e.g., \cite[Exercise I.1.14]{lam}.
\end{proof}

One measure of the complexity of quadratic forms over a field $k$ is the $u$-\textit{invariant} of $k$, denoted by $u(k)$, which is defined to be the maximal dimension of an anisotropic quadratic form over $k$. If no such maximum exists, we say that $u(k) = \infty$ \cite[Definition~XI.6.1]{lam}. In general, computing the $u$-invariant of a given field is a challenging problem. In Subsection \ref{counterexamples}, we will work with fields~$\ell$ of known $u$-invariant. Indeed, in Subsection \ref{counterexamples} we consider fields $\ell$ that satisfy property~$\mathscr{A}_i(2)$ for some $i \geq 0$ (such fields were also studied in, e.g., \cite{cas, leep}). Recall that, for any integer $i \geq 0$, a field $\ell$ of characteristic $\ne 2$ satisfies property $\mathscr{A}_i(2)$ (written $\ell \in \mathscr{A}_i(2)$) if every system of~$s$ quadratic forms over $\ell$ in $n > s \cdot 2^i$ common variables has a non-trivial simultaneous zero in an odd degree extension of $\ell$. By \cite[Proposition~2.2]{leep}, if $\ell \in \mathscr{A}_i(2)$, then $u(\ell) \leq 2^i$. Moreover, if $\ell \in \mathscr{A}_i(2)$ for some $i \geq 0$ and $L$ is a finitely generated transcendence degree $r$ extension of $\ell$, then $L \in \mathscr{A}_{i+r}(2)$ \cite[Theorems~2.3, 2.5]{leep}, hence $u(L) \leq 2^{i+r}$.

The next result, which will be used at various points of this article, shows how the $u$-invariant can be used to study the Witt index.
\begin{lemma}
\label{Witt index of large dimensional forms}
Let $k$ be a field of characteristic $\ne 2$ with $u(k) < \infty$. For any integer $j \geq 1$ and any quadratic form $q$ over $k$, if $\dim q \geq u(k) + 2j - 1$, then $i_W(q) \geq j$.
\end{lemma}
\begin{proof}
We prove the lemma by induction on $j \geq 1$. For the base case $j = 1$, let $q$ be any quadratic form over $k$ with $\dim q \geq u(k) + 2 - 1 = u(k) + 1$. Then because $\dim q > u(k)$, $q$ must be isotropic, hence $i_W(q) \geq 1$.

Now assume for some $j \geq 1$ that all quadratic forms over $k$ of dimension $\geq u(k) + 2j - 1$ have Witt index at least $j$, and let $q$ be any quadratic form over $k$ of dimension $\geq u(k) + 2(j+1) - 1$. Then because $\dim q > u(k)$, $q$ must be isotropic and we can write $q \simeq \HH \perp q'$ for some form $q'$ over~$k$ with $\dim q' \geq u(k) + 2j - 1$. By the induction hypothesis, $i_W(q') \geq j$, thus $i_W(q) \geq j + 1$, proving the claim by induction.
\end{proof}

Throughout this manuscript we will be considering quadratic forms over complete discretely valued fields. For a field $k$ equipped with a discrete valuation $v$, we let $\mathcal{O}_v$ denote the valuation ring of $v$ with maximal ideal $\mathfrak{m}_v$. We let $\kappa_v = \mathcal{O}_v / \mathfrak{m}_v$ be the residue field, and let $k_v$ be the $v$-adic completion of $k$. For a quadratic form $q$ over $k$, we write $q_v$ for the quadratic form $q_{k_v}$. Over the field $K= k(t)$, each monic irreducible polynomial $\pi \in k[t]$ induces a discrete valuation $v_{\pi}$ on $K$, and we will denote the residue field of $v_\pi$ by $\kappa_{\pi} \cong k[t] / (\pi)$.

A particularly useful tool when studying quadratic forms over complete discretely valued fields is Springer's Theorem \cite[Proposition~VI.I.9]{lam}. Springer's Theorem states that, over a complete discretely valued field $K$ with uniformizer $\pi$, valuation ring $\mathcal{O}_K$, and residue field~$\kappa$ with characteristic $\ne 2$, a quadratic form $q \simeq q_1 \perp \pi \cdot q_2$ (where the entries of $q_1$ and $q_2$ are all units in~$\mathcal{O}_K$) is anisotropic over $K$ if and only if both residue forms $\overline{q}_1$ and $\overline{q}_2$ are anisotropic over $\kappa$.

At various points of this paper we will need to use another local-global principle for quadratic forms: the local-global principle for isometry. Given a pair of quadratic forms $q_1$, $q_2$ over a field~$k$ of characteristic $\ne 2$ equipped with a non-empty set $V$ of non-trivial discrete valuations, we say that~$q_1$ and $q_2$ satisfy the \textit{local-global principle for isometry with respect to $V$} if $q_1$ and $q_2$ being isometric over $k_v$ for all $v \in V$ implies that $q_1$ and $q_2$ are isometric over $k$. Two quadratic forms $q_1, q_2$ over~$k$ are isometric if and only if $q_1$ and $q_2$ have the same dimension and $q_1 \perp -q_2$ is hyperbolic. So the local-global principle for isometry with respect to $V$ is satisfied over $k$ if and only if every even-dimensional quadratic form over $k$ that is hyperbolic over $k_v$ for all $v \in V$ is also hyperbolic over $k$. If the local-global principle for isotropy is satisfied by all quadratic forms over $k$, then the local-global principle for isometry is satisfied as well. However, there are fields~$k$ equipped with sets $V$ of discrete valuations with respect to which the local-global principle for isometry holds, but there exist quadratic forms over $k$ that violate the local-global principle for isotropy with respect to $V$. For instance, for any field $k$ of characteristic $\ne 2$, over $K = k(t)$ the local-global principle for isometry holds with respect to the set $V_{K/k}$ of non-trivial discrete valuations on $K$ that are trivial on~$k$ (see, e.g., \cite[Proposition~2.1a)]{cas}). But for various fields $k$, there are counterexamples to the local-global principle for isotropy over $K = k(t)$ with respect to $V_{K/k}$ (see, e.g., \cite{auel-suresh, cas, gup18, gup21}).

In Subsection \ref{I^n-neighbors} we will consider quadratic forms that belong to powers $I^n(k)$ of the fundamental ideal $I(k)$ of even-dimensional quadratic forms in the Witt ring $W(k)$ of a field $k$ of characteristic $\ne 2$. Given a quadratic form $q$ over $k$, we will write $q \in I^n(k)$ if the Witt equivalence class $[q]$ of $q$ belongs to $I^n(k)$. Recall that, for any integer $n \geq 1$, $I^n(k)$ is additively generated by $n$-fold Pfister forms; i.e., by quadratic forms of the form $\langle 1, a_1 \rangle \otimes \cdots \otimes \langle 1, a_n \rangle =: \langle \langle a_1, \ldots, a_n \rangle \rangle$, where $a_1, \ldots, a_n \in k^{\times}$.

\section{Refined local-global principle for isotropy}
\label{LGP}
Let $k$ be a field of characteristic $\ne 2$ equipped with a non-empty set $V$ of non-trivial discrete valuations. We can phrase the local-global principle for isotropy with respect to $V$ in terms of the Witt index. Indeed, a quadratic form $q$ over $k$ satisfies the local-global principle for isotropy with respect to $V$ if and only if $i_W(q_v) \geq 1 \text{ for all $v \in V$ implies } i_W(q) \geq 1$. We can therefore use the Witt index to define a refined local-global principle for isotropy as follows.
\begin{defin}
\label{lgp(r,s)}
Let~$k$ be a field of characteristic~$\ne 2$, let~$V$ be a non-empty set of non-trivial discrete valuations on~$k$, and let~$r, s \geq 1$ be positive integers. We say that a quadratic form~$q$ over~$k$ \textit{satisfies $\emph{LGP}(r,s)$ with respect to $V$} if 
\[
	i_W(q_v) \geq r \text{ for all $v \in V$ implies } i_W(q) \geq s.
\]
\end{defin}
In particular,~$q$ satisfies the local-global principle for isotropy if and only if it satisfies $\lgp(1,1)$.

\begin{remark}
Let~$k$ be a field of characteristic~$\ne 2$ equipped with a non-empty set~$V$ of non-trivial discrete valuations. Let~$q$ be a quadratic form over~$k$ and let~$r, s \geq 1$ be any positive integers. For any integer~$r' > r$, if~$q$ satisfies~$\lgp(r, s)$ with respect to~$V$, then~$q$ satisfies~$\lgp(r', s)$ with respect to~$V$ as well, and for any integer~$s' > s$, if~$q$ satisfies~$\lgp(r, s')$ with respect to~$V$, then it also satisfies~$\lgp(r, s)$ with respect to~$V$.
\end{remark}
This section is organized as follows. In Subsection \ref{counterexamples} we show that there are numerous counterexamples to $\lgp(r,1)$ for various integers $r$ over purely transcendental field extensions of fields $\ell \in \mathscr{A}_i(2)$ for some $i \geq 0$ (see Theorem \ref{ce to lgp}). Inspired by the existence of these counterexamples, in Subsection \ref{I^n-neighbors} we find a certain condition that a quadratic form $q$ can satisfy to guarantee that $q$ satisfies $\lgp(r,s)$ for some integers $r, s \geq 1$. This condition is that of being an $I^n$-\textit{neighbor} (see Definition~\ref{I^n neighbors}). To conclude, in Subsection \ref{all forms} we investigate whether there are integers $r, s \geq 1$ such that \textit{all} quadratic forms over a field $k$ satisfy $\lgp(r,s)$ with respect to a non-empty set $V$ of non-trivial discrete valuations on $k$.
\subsection{Counterexamples}
\label{counterexamples}
The main goal of this subsection is to prove the following result.
\begin{theorem}
\label{ce to lgp}
Let $\ell$ be a field of characteristic $\ne 2$. Assume that $\ell \in \mathscr{A}_i(2)$ for some $i \geq 0$ and $u(\ell) = 2^i$. For any integer $r \geq 1$ let $L_r = \ell(x_1, \ldots, x_r)$, and for $r \geq 2$ let $V_r$ be the set of non-trivial discrete valuations on $L_r$ that are trivial on $L_{r-1}$. Then for $r \geq 2$ and any integer $n$ such that $0 \leq n < 2^{i+r-2}$, there exists a $\left(2^{i+r}-n\right)$-dimensional quadratic form over $L_r$ that violates $\emph{LGP}\left(2^{i+r-2} - n, 1\right)$ with respect to $V_r$.
\end{theorem}
Theorem \ref{ce to lgp} provides a partial generalization of \cite[Theorem~2.3]{cas}. Indeed, in the context of Theorem \ref{ce to lgp}, \cite[Theorem~2.3]{cas} states that there are counterexamples over $L_r$ to $\lgp(1,1)$ with respect to $V_r$ in dimensions $2^{i+r-1} + 1$ to $2^{i+r}$. Theorem \ref{ce to lgp} shows the existence of counterexamples in half as many dimensions as \cite[Theorem~2.3]{cas}, but these counterexamples all have local Witt index at least 1 (and in most cases have local Witt index strictly larger than 1).

The idea of the proof of Theorem \ref{ce to lgp} is as follows. We will first construct a $2^{i+r}$-dimensional quadratic form over $L_r$ that violates $\lgp\left(2^{i+r-2}, 1\right)$ with respect to $V_r$. Once we have constructed this form, if we take any $\left(2^{i+r}-n\right)$-dimensional subform, then Lemma \ref{Witt index of sum} implies that this subform violates $\lgp\left(2^{i+r-2} - n, 1 \right)$ with respect to~$V_r$.

\begin{lemma}
\label{anisotropic tensor product}
Let $k$ be any field of characteristic $\ne 2$, and let $q$ be any anisotropic quadratic form over $k$. Then over $k(x_1, x_2)$, the quadratic form
\[
	\varphi = \langle x_2 + 1, -x_1 - x_2, x_1, x_1x_2 \rangle \otimes q 
\]
is anisotropic.
\end{lemma}
\begin{proof}
We prove the stronger claim that $\varphi$ is anisotropic over $k(x_2)((x_1))$. Over $k(x_2)((x_1))$, we have
\[
	\varphi = \left(\langle x_2 + 1, -x_1 - x_2 \rangle \otimes q\right) \perp x_1 \cdot \left(\langle 1, x_2 \rangle \otimes q \right).
\]
The second residue form of $\varphi$, $\langle 1, x_2 \rangle \otimes q$, is anisotropic over the residue field $k(x_2)$ by \cite[Lemma~2.4]{cas}. The first residue form of $\varphi$ is
\[
	\varphi_1 = \langle x_2 + 1, -x_2 \rangle \otimes q = (x_2 + 1) \cdot q \perp x_2 \cdot (-q).
\]
Considering the form $\varphi_1$ over $k((x_2))$, the first residue form of $\varphi_1$ is $q$ and the second residue form is $-q$, both of which are anisotropic over $k$ by assumption. Therefore $\varphi_1$ is anisotropic over $k((x_2)) \supset k(x_2)$ by Springer's Theorem. Thus both residue forms of $\varphi$ are anisotropic over $k(x_2)$, so by Springer's Theorem $\varphi$ is anisotropic over $k(x_2)((x_1)) \supset k(x_1, x_2)$, proving the claim.
\end{proof}

\begin{lemma}
\label{local isotropy}
Let $\ell$ be a field of characteristic $\ne 2$ such that $\ell \in \mathscr{A}_i(2)$ for some $i \geq 0$ and $u(\ell) = 2^i$. Let $q$ be a $2^i$-dimensional anisotropic quadratic form over $\ell$. Then for all non-trivial discrete valuations $v$ on $\ell(x_1, x_2)$ that are trivial on $\ell(x_1)$, the $2^{i+2}$-dimensional quadratic form defined over $\ell(x_1, x_2)$ by 
\[
	\varphi = \langle x_2 + 1, -x_1 - x_2, x_1, x_1x_2 \rangle \otimes q
\]
satisfies $i_W(\varphi_v) \geq 2^i$.
\end{lemma}
\begin{proof}
Before we prove the lemma, we recall \cite[Theorem~2.1.4(b)]{valued fields}. Let $K = k(t)$ for any field $k$, let $\mathscr{P}$ be the set of all monic irreducible polynomials in $k[t]$, and let $V_{K/k}$ be the set of all non-trivial discrete valuations on $K$ that are trivial on $k$. Then $V_{K/k} = \{v_{\pi} \mid \pi \in \mathscr{P}\} \cup \{v_{\infty}\}$, where $v_{\infty}$ is the degree valuation with respect to $t$. We now prove the lemma by considering several cases for $v$. 

\underline{Case 1}: $v = v_{\infty}$ is the degree valuation with uniformizer $x_2^{-1}$.

We can write $\langle x_2 + 1, -x_1 - x_2 \rangle = x_2 \cdot \left\langle 1 + x_2^{-1}, -x_1x_2^{-1} - 1 \right\rangle$. Scaling this form by $x_2^{-2}$, we have
\[
	\langle x_2 + 1, -x_1 - x_2 \rangle \simeq x_2^{-1} \cdot \left\langle 1 + x_2^{-1}, -x_1x_2^{-1} - 1 \right\rangle.
\]
The second residue form of this quadratic form is $\langle 1, -1 \rangle$, which is isotropic over the residue field~$\ell(x_1)$. So $\langle x_2 + 1, -x_1-x_2 \rangle$ is isotropic over $\ell(x_1, x_2)_{v_{\infty}}$ by Springer's Theorem. Therefore, since $\dim q = 2^i$, \cite[Corollary~I.6.1]{lam} implies that
\[
	i_W((\langle x_2 + 1, -x_1 - x_2 \rangle \otimes q)_{v_{\infty}}) \geq 2^i.
\]
Thus $i_W(\varphi_{v_{\infty}}) \geq 2^i$ since $\varphi$ contains $\langle x_2 + 1, -x_1 - x_2 \rangle \otimes q$ as a subform.

\underline{Case 2}: $v = v_{\pi}$, where $\pi = x_2, x_2 + 1, x_1 + x_2$ is a divisor of at least one entry of $\varphi$.

In this case, the quadratic forms $\langle -x_1 - x_2, x_1 \rangle$, $\langle x_1, x_1x_2 \rangle$, and $\langle x_2 + 1, x_1, x_1x_2 \rangle$ each reduce to isotropic quadratic forms over the respective residue field $\kappa_{\pi}$. By the same argument as in Case 1, since $\dim q = 2^i$ we have $i_W(\varphi_{v_{\pi}}) \geq 2^i$.

\underline{Case 3}: $v = v_{\pi}$ for a monic irreducible polynomial $\pi \in \ell(x_1)[x_2]$ different from $x_2, x_2 + 1, x_1 + x_2$.

In this case, each entry of $\varphi$ is a unit in $\mathcal{O}_{v_{\pi}}$, so $\varphi$ reduces to a $2^{i+2}$-dimensional quadratic form $\overline{\varphi}$ over the residue field $\kappa_{\pi}$. The field $\kappa_{\pi}$ is a finite extension of $\ell(x_1) \in \mathscr{A}_{i+1}(2)$, therefore $u(\kappa_{\pi}) \leq 2^{i+1}$ (see Section \ref{notation}). The form $\overline{\varphi}$ has dimension $2^{i+2} = 2^{i+1} + 2 \cdot 2^i \geq u(\kappa_{\pi}) + 2 \cdot 2^i$, so $i_W\left(\overline{\varphi}\right) \geq 2^i$ by Lemma \ref{Witt index of large dimensional forms}. This implies, by Springer's Theorem, that $i_W(\varphi_{v_{\pi}}) \geq 2^i$.

The first paragraph of the proof shows that these three cases above cover all possibilities for $v$, so the proof is complete.
\end{proof}

We now prove Theorem \ref{ce to lgp}.

\begin{proof}[Proof of Theorem \ref{ce to lgp}]
We first consider the case $r = 2$. Let $q$ be an anisotropic $2^i$-dimensional quadratic form over $\ell$ and consider the $2^{i+2}$-dimensional quadratic form
\[
	\varphi_2 = \langle x_2 + 1, -x_1 - x_2, x_1, x_1x_2 \rangle \otimes q 
\]
over $L_2 = \ell(x_1, x_2)$. By Lemma \ref{anisotropic tensor product}, $\varphi_2$ is anisotropic over $L_2$, and by Lemma \ref{local isotropy}, $i_W(\varphi_{2,v}) \geq 2^i$ for all $v \in V_2$. Therefore $\varphi_2$ violates $\lgp\left(2^i, 1\right)$ with respect to $V_2$. For any $n$ such that $0 < n < 2^i$, let $\psi_n$ be any $\left(2^{i+2} - n \right)$-dimensional subform of $\varphi_2$. Then $\psi_n$ is anisotropic over $L_2$, and for all $v \in V_2$, Lemma \ref{Witt index of sum} implies that $i_W(\psi_{n,v}) \geq 2^i - n$. Hence $\psi_n$ violates $\lgp\left(2^i - n, 1 \right)$ with respect to $V_2$, completing the proof if $r = 2$.

Now suppose $r \geq 3$. We have $L_r = \ell(x_1, \ldots, x_r) \cong L_{r-2}(x_{r-1}, x_r)$, and since $\ell \in$~$\mathscr{A}_i(2)$, it follows that $L_{r-2} \in \mathscr{A}_{i+r-2}(2)$. Moreover, if $q$ is an anisotropic $2^i$-dimensional quadratic form over $\ell$, then the $2^{i+r-2}$-dimensional quadratic form $q \otimes \langle \langle x_1, \ldots, x_{r-2} \rangle \rangle$ is anisotropic over $L_{r-2}$ \cite[Lemma~2.4]{cas}. So $u(L_{r-2}) = 2^{i+r-2}$, and by applying Lemmas \ref{anisotropic tensor product} and \ref{local isotropy} to the $2^{i+r}$-dimensional quadratic form
\[
	\varphi_r = \langle x_r + 1, -x_{r-1} - x_r, x_{r-1}, x_{r-1}x_r \rangle \otimes \left(q \otimes \langle \langle x_1, \ldots, x_{r-2} \rangle \rangle\right)
\]
over $L_r$, we conclude that $\varphi_r$ violates $\lgp\left(2^{i+r-2}, 1 \right)$ with respect to $V_r$. For any $n$ such that $0 < n < 2^{i+r-2}$, all $\left(2^{i+r} - n \right)$-dimensional subforms of $\varphi_r$ violate $\lgp\left(2^{i+r-2} - n, 1 \right)$ with respect to $V_r$, completing the proof for $r \geq 3$.
\end{proof}
\begin{examples}
The following are special cases of Theorem \ref{ce to lgp}.
\begin{enumerate}[label=(\arabic*)]
    \item For any prime $p \ne 2$, the field $\F_p \in \mathscr{A}_1(2)$ and $u(\F_p) = 2$. Thus over $\F_p(x_1, x_2)$ there is an eight-dimensional quadratic form that violates $\lgp(2,1)$ with respect to $V_2$, and a seven-dimensional quadratic form that violates $\lgp(1,1)$ with respect to $V_2$.

    \item For any prime $p \ne 2$, $\Q_p \in \mathscr{A}_2(2)$ \cite[Corollary~2.7]{leep} and $u(\Q_p) = 4$. Thus over $\Q_p(x_1, x_2)$ there are quadratic forms that violate $\lgp(r, 1)$ with respect to $V_2$ for $r = 1, 2, 3, 4$.
\end{enumerate}
\end{examples}

\subsection{$I^n$-neighbors}
\label{I^n-neighbors}
As we saw in Subsection \ref{counterexamples}, over purely transcendental field extensions $L$ of fields $\ell \in \mathscr{A}_i(2)$ for some $i \geq 0$, there is a set $W$ of discrete valuations on $L$ with respect to which there are numerous counterexamples to $\lgp(r, 1)$ despite the fact that the local-global principle for isometry is satisfied over $L$ with respect to $W$. We are then naturally led to ask if there are certain conditions we can impose on a quadratic form $q$ over a field $k$ equipped with a set $V$ of discrete valuations to ensure that $q$ satisfies $\lgp(r, s)$ with respect to $V$ for some integers $r, s \geq 1$. In this subsection, we explore such a condition on quadratic forms over a field $k$ equipped with a set~$V$ of discrete valuations with respect to which the local-global principle for isometry holds (see Proposition \ref{I^n neighbors and lgp}).

We first recall that if $k$ is a field equipped with a non-empty set $V$ of non-trivial discrete valuations with respect to which the local-global principle for isometry holds, then Pfister neighbors over~$k$ satisfy the local-global principle for isotropy with respect to $V$ (see, e.g., \cite[Remark~2.2a)]{cas}). Recall that a quadratic form $q$ over $k$ is a \textit{Pfister neighbor} if it is similar to a subform of a Pfister form $\varphi$ over $k$ with $\dim \varphi < 2 \dim q$ (see, e.g., \cite[Definition~X.4.16]{lam}). Now, $n$-fold Pfister forms over $k$ additively generate $I^n(k)$, the $n$-th power of the fundamental ideal of the Witt ring~$W(k)$. So a Pfister neighbor $q$ is a subform of some $2^n$-dimensional form $\varphi \in I^n(k)$. Here we recall that by $\varphi \in I^n(k)$ we mean that the Witt equivalence class $[\varphi]$ of $\varphi$ belongs to $I^n(k)$. This motivates the following generalized definition.
\begin{defin}
\label{I^n neighbors}
Let $n \geq 1$ be a positive integer. A quadratic form $q$ over a field $k$ is an \textit{$I^n$-neighbor of complementary dimension $r$} if there exists an $r$-dimensional quadratic form $\sigma_r$ over $k$, with $0 \leq r < \dim q$, such that $q \perp \sigma_r \in I^n(k)$. The form $\sigma_r$ is called a \textit{complementary form} of~$q$.
\end{defin}
\begin{remark}
The complementary dimension of an $I^n$-neighbor $q$ over a field $k$ is not always unique. Indeed, if $q \perp \sigma_r \in I^n(k)$ and $r < \dim q - 2$, then $q \perp \sigma_r \perp \HH \in I^n(k)$. However, the parity of the complementary dimension of the $I^n$-neighbor is unique. Indeed, $\dim q$ and any complementary dimension $r$ must have the same parity since any form in $I^n(k)$ for $n \geq 1$ must have even dimension.
\end{remark}

\begin{examples} Let $k$ be any field of characteristic $\ne 2$.
\begin{enumerate}[label=(\arabic*)]

    \item If $q$ is a quadratic form over $k$ such that $q \in I^n(k)$ for some $n \geq 1$, then $q$ is an $I^n$-neighbor of complementary dimension 0.
    
	\item All quadratic forms $q$ over $k$ of dimension $\geq 2$ are $I^1$-neighbors, as a quadratic form belongs to $I^1(k)$ if and only if it has even dimension. So either $q$ or $q \perp \langle 1 \rangle$ belongs to $I^1(k)$.
	
	\item All quadratic forms $q$ over $k$ of dimension $\geq 3$ are $I^2$-neighbors. Indeed, a quadratic form  belongs to $I^2(k)$ if and only if it has even dimension and trivial signed determinant. So either $q \perp \langle \pm \det q \rangle$ or $q \perp \langle 1, \pm \det q \rangle$ belongs to $I^2(k)$, with the sign of $\det q$ chosen based on $\dim q$.
	
	\item If $q$ is a Pfister neighbor over $k$ whose associated Pfister form $\varphi$ has dimension $2^n$, then $q$ is an $I^n$-neighbor.

    \item If $q$ is a twisted $(n,m)$-Pfister form for some $1 \leq m < n$ (see \cite[Definition~3.4(i)]{hoffmann} for the definition), then $q$ is an $I^n$-neighbor of complementary dimension $2^m$.
\end{enumerate}
\end{examples}

We now prove some preliminary results about $I^n$-neighbors. In our study of $I^n$-neighbors we will frequently use a result of Arason and Pfister that states that if a quadratic form $q$ over a field~$k$ of characteristic $\ne 2$ belongs to $I^n(k)$ for some $n \geq 1$ and $\dim q < 2^n$, then $q$ is hyperbolic (see, e.g, \cite[Hauptsatz~X.5.1]{lam}).

\begin{lemma}
\label{I^n neighbors of small dimension}
If a quadratic form $q$ over a field $k$ of characteristic $\ne 2$ is an $I^n$-neighbor of complementary dimension $r$ and $2^n > \dim q + r$, then $i_W(q) \geq \frac{\dim q - r}{2}$. In particular, $q$ is isotropic.
\end{lemma}
\begin{proof}
By assumption, there exists an $r$-dimensional form $\sigma_r$ over $k$ such that $q \perp \sigma_r \in I^n(k)$. Since $2^n > \dim q + r$ by assumption, \cite[Hauptsatz~X.5.1]{lam} implies that $q \perp \sigma_r$ is hyperbolic, i.e., $q \perp \sigma_r \simeq \left(\frac{\dim q + r}{2}\right) \HH$. Then since $\dim q = \frac{\dim q + r}{2} + \frac{\dim q - r}{2}$ and $\dim q - r > 0$, we conclude by Lemma \ref{subform of hyperbolic form} that $i_W(q) \geq \frac{\dim q - r}{2} \geq 1$, as desired.
\end{proof}

\begin{prop}
\label{isometric complementary forms}
Suppose that $q$ is an $I^n$-neighbor of complementary dimension $r$ over a field $k$ of characteristic $\ne 2$. If $r < 2^{n-1}$, then all $r$-dimensional complementary forms of $q$ are isometric.
\end{prop}
\begin{proof}
Let $\sigma_r, \sigma_r'$ be $r$-dimensional forms over $k$ such that $q \perp \sigma_r, q \perp \sigma_r' \in I^n(k)$. Then
\[
	\varphi = (q \perp \sigma_r) \perp -(q \perp \sigma_r') \in I^n(k).
\]
The form $\varphi$ is Witt equivalent to $\sigma_r \perp -\sigma_r'$, thus $\sigma_r \perp -\sigma_r' \in I^n(k)$. By assumption, $r < 2^{n-1}$, so $\dim (\sigma_r \perp -\sigma_r') < 2^n$. By \cite[Hauptsatz~X.5.1]{lam}, this implies that $\sigma_r \perp -\sigma_r \simeq r\HH$, so $\sigma_r$ and~$\sigma_r'$ must be isometric.
\end{proof}
\begin{remarks}
Let $k$ be any field of characteristic $\ne 2$.
\begin{enumerate}[label=(\arabic*)]
	\item If $q$ is an $I^n$-neighbor of complementary dimension $r$ over $k$ and $2^n > \dim q + r$, then Proposition \ref{isometric complementary forms} shows that the form $\sigma_r$ over $k$ such that $q \perp \sigma_r \in I^n(k)$ is unique up to isometry. Indeed, since $\dim q > r$ we have $2^n > \dim q + r > 2r$. Hence $2^{n-1} > r$, so Proposition \ref{isometric complementary forms} applies.
	
	\item If $q$ is an $I^n$-neighbor of complementary dimension $r = 2^{n-1}$ over $k$, then any two $r$-dimensional complementary forms of $q$ that represent a common element of $k$ must be isometric. Indeed, suppose $\sigma_r$ and $\sigma_r'$ are $r$-dimensional complementary forms of $q$ that represent a common element of $k$. Then, as we saw in the proof Proposition \ref{isometric complementary forms}, $\sigma_r \perp -\sigma_r' \in I^n(k)$. Moreover, since $\sigma_r$ and $\sigma_r'$ represent a common element of $k$, the form $\sigma_r \perp -\sigma_r'$ is isotropic \cite[Corollary~I.3.6]{lam}. So $\sigma_r \perp -\sigma_r'$ is Witt equivalent to a form in $I^n(k)$ with dimension~$< 2^n$, and therefore must be hyperbolic \cite[Hauptsatz~X.5.1]{lam}. Thus $\sigma_r \simeq \sigma_r'$.
\end{enumerate}
\end{remarks}
The following example shows that complementary forms of an $I^n$-neighbor need not be isometric if the conditions of Proposition \ref{isometric complementary forms} are not met.
\begin{example}
Let $k$ be any field of characteristic $\ne 2$, and let $K = k(x, y)$. Let $q$ be any four-dimensional quadratic form over $K$ with determinant $-x$. Then $q \perp \langle 1, x \rangle$ and $q \perp \langle -y, -xy \rangle$ both belong to $I^2(K)$. However, $\langle 1, x \rangle \not\simeq \langle -y, -xy \rangle$ since the Pfister form $\langle \langle x, y \rangle \rangle$ is anisotropic over $K$. 
\end{example}
We will now show that, given an $I^n$-neighbor $q$ over a field $k$ of characteristic $\ne 2$ equipped with a non-empty set $V$ of non-trivial discrete valuations with respect to which the local-global principle for isometry holds, we can find integers $r, s \geq 1$ such that $q$ satisfies $\lgp(r, s)$ with respect to~$V$ (see Proposition \ref{I^n neighbors and lgp}). Before proving this result, we collect some results about the behavior of quadratic forms in $I^n(k)$.

Recall that, for a field $k$ of characteristic $\ne 2$ and a positive integer $n \geq 1$, GP$_n(k)$ is the set of quadratic forms $q$ over $k$ such that $q \simeq a \cdot \varphi$ for some $a \in k^{\times}$ and some $n$-fold Pfister form $\varphi$ over~$k$ (see, e.g., \cite{hoffmann}). Then, by \cite[Theorem~X.5.6]{lam},
\[
	\text{GP}_n(k) = \left\{\text{quadratic forms $q$ over $k$} \mid \dim q = 2^n \text{ and } q \in I^n(k) \right\}.
\]
\begin{lemma}
\label{Witt equivalent to GP}
Let $q$ be an even-dimensional quadratic form over a field $k$ of characteristic $\ne 2$ with $\dim q \geq 2^n$ for some integer $n \geq 1$. If $q$ is Witt equivalent to an isotropic form $\psi \in \emph{GP}_n(k)$, then $q$ is hyperbolic.
\end{lemma}
\begin{proof}
Since $\psi \in \text{GP}_n(k)$, we have $\dim \psi = 2^n \leq \dim q$. Therefore, because $q$ is Witt equivalent to~$\psi$, it follows that $q \simeq \psi \perp \frac{\dim q - \dim \psi}{2}\HH$. By definition, because $\psi \in \text{GP}_n(k)$, there must be some $a \in k^{\times}$ and some $n$-fold Pfister form $\varphi$ over~$k$ such that $\psi \simeq a \cdot \varphi$. Furthermore, by assumption,~$\psi$ is isotropic, thus $\varphi$ is isotropic as well. Since~$\varphi$ is an isotropic Pfister form, $\varphi$ must be hyperbolic \cite[Theorem~X.1.7]{lam}. Therefore~$\psi$ is hyperbolic, which implies that $q$ is hyperbolic as well.
\end{proof}
We can now prove
\begin{prop}
\label{I^n neighbors and lgp}
Let $k$ be a field of characteristic $\ne 2$ equipped with a non-empty set $V$ of non-trivial discrete valuations with respect to which the local-global principle for isometry holds. Let $q$ be an $I^n$-neighbor of complementary dimension $r$ over $k$ for some $n \geq 1$. Then
\[
	q \text{ satisfies } \emph{LGP}\left(\frac{\dim q + r - 2^n}{2} + 1, \frac{\dim q - r}{2}\right) \text{ with respect to $V$}.
\]
\end{prop}
\begin{proof}
Because $q$ is an $I^n$-neighbor of complementary dimension $r$, there exists an $r$-dimensional form $\sigma_r$ over $k$ such that $q \perp \sigma_r \in I^n(k)$. We note that, to show $i_W(q) \geq \frac{\dim q - r}{2}$, it suffices to show that $q \perp \sigma_r$ is hyperbolic. Indeed, suppose $q \perp \sigma_r \simeq \frac{\dim q + r}{2} \HH$.
Since $\dim q = \frac{\dim q + r}{2} + \frac{\dim q - r}{2}$ and $\dim q - r > 0$, by Lemma \ref{subform of hyperbolic form} we have $i_W(q) \geq \frac{\dim q - r}{2}$, as desired. So we will show that $i_W(q_v) \geq \frac{\dim q + r - 2^n}{2} + 1$ for all $v \in V$ implies that $q \perp \sigma_r$ is hyperbolic over $k$.

We first observe that if $\dim q + r < 2^n$, then by Lemma \ref{I^n neighbors of small dimension}, $i_W(q) \geq \frac{\dim q - r}{2}$ without any assumptions on the Witt index of $q$ over $k_v$.

So assume that $\dim q + r \geq 2^n$. By our assumption on the Witt index of $q$ over $k_v$, we have 
\[
	i_W((q \perp \sigma_r)_v) \geq \frac{\dim q + r - 2^n}{2} + 1
\]
for all $v \in V$. Thus for each $v \in V$ we can find a $2^n$-dimensional quadratic form $q'_v$ over $k_v$ such that 
\[
	(q \perp \sigma_r)_v \simeq \frac{\dim q + r - 2^n}{2} \HH \perp q'_v,
\]
where $i_W(q'_v) \geq 1$. Now, $q \perp \sigma_r \in I^n(k)$, so $(q \perp \sigma_r)_v \in I^n(k_v)$. Therefore $q'_v \in I^n(k_v)$, hence $q'_v \in \text{GP}_n(k_v)$ by \cite[Theorem~X.5.6]{lam}. So $(q \perp \sigma_r)_v$ is Witt equivalent to the isotropic form $q'_v \in \text{GP}_n(k_v)$, thus $(q \perp \sigma_r)_v$ is hyperbolic by Lemma \ref{Witt equivalent to GP}. This holds for all $v \in V$, so $q \perp \sigma_r$ is hyperbolic over~$k$ since the local-global principle for isometry holds with respect to $V$. We have shown that $q \perp \sigma_r$ is hyperbolic over $k$, so the proof is complete.
\end{proof}

\begin{remark}
\label{Witt index increases}
In the context of Proposition \ref{I^n neighbors and lgp}, if $q$ is an $I^n$-neighbor of complementary dimension $r < 2^{n-1} - 1$, then
\[
	\frac{\dim q + r - 2^n}{2} + 1 < \frac{\dim q - r}{2}.
\]
Therefore, somewhat counterintuitively, Proposition \ref{I^n neighbors and lgp} shows that for such an $I^n$-neighbor, the Witt index \textit{increases} as we pass from the local setting to the global setting.
\end{remark}

Proposition \ref{I^n neighbors and lgp} shows that if a quadratic form $q$ over a field $k$ of characteristic $\ne 2$ is an $I^n$-neighbor for some $n$, then there are integers $r, s \geq 1$ such that $q$ satisfies $\lgp(r,s)$. For large $n$, however, it is challenging to determine when a quadratic form belongs to $I^n(k)$, so it is difficult to verify that $q$ is an $I^n$-neighbor. Despite this, for $n \leq 3$, this verification can be done relatively easily. The fundamental ideal $I(k)$ consists of quadratic forms of even dimension, and $I^2(k)$ consists of even-dimensional quadratic forms with trivial discriminant. It is therefore straightforward to check if a quadratic form belongs to $I(k)$ or $I^2(k)$, so it is easy to check if $q$ is an $I^n$-neighbor for $n = 1, 2$.

For $n = 3$, the Witt and Hasse invariants of $q$, denoted by $c(q)$ and $s(q)$, respectively, allow us to easily verify whether a quadratic form belongs to $I^3(k)$. Indeed, a quadratic form $q$ over $k$ of dimension $2r$ belongs to $I^3(k)$ if and only if $\det q = (-1)^r$ and $c(q) = 1 \in \Br(k)$ \cite[p.~138]{lam}, where $\Br(k)$ denotes the Brauer group of $k$. We now briefly recall the definitions of these invariants $c(q)$ and $s(q)$ (see, e.g., \cite[Chapter~V.3]{lam}).

Given an $n$-dimensional quadratic form $q \simeq \langle a_1, \ldots, a_n \rangle$ over $k$, let $C(q)$ denote the Clifford algebra of $q$. The Clifford algebra $C(q)$ has a $\mathbb{Z} / 2 \mathbb{Z}$-grading, and we let $C_0(q)$ denote the ``even part'' of $C(q)$. If $n$ is even, then $C(q)$ is a central simple $k$-algebra, and if $n$ is odd, $C_0(q)$ is a central simple $k$-algebra. The \textit{Witt invariant} of $q$, denoted by $c(q)$, is then defined by
\[
	c(q) = \begin{cases}
		[C(q)] &\text{ if $n$ is even}, \\
		[C_0(q)] &\text{ if $n$ is odd},
	\end{cases}
\]
where $[C(q)]$, $[C_0(q)]$ denote the classes of these central simple $k$-algebras in $\Br(k)$. 

The Hasse invariant of $q \simeq \langle a_1, \ldots, a_n \rangle$ has a more straightforward definition using quaternion algebras. For any $a, b \in k^{\times}$, we let $\left(\frac{a, b}{k}\right)$ denote the generalized quaternion algebra over $k$, which is the central simple $k$-algebra generated by $i, j$ such that $i^2 = a, j^2 = b$, and $ij = -ji$. The \textit{Hasse invariant} of $q$, denoted by $s(q)$, is defined by the class of
\[
	\prod_{1 \leq i < j \leq n} \left(\frac{a_i, a_j}{k}\right)
\]
in $\Br(k)$. If $n = 1$, we take this product to be 1. 

The Hasse invariant is more amenable to computations than the Witt invariant, particularly because, for quadratic forms $q_1, q_2$ over $k$, we have (see, e.g., \cite[p.~119]{lam})
\[
	s(q_1 \perp q_2) = s(q_1)s(q_2) \left(\frac{\det q_1, \det q_2}{k}\right).
\]
However, these two invariants are closely related to one another via \cite[Proposition~V.3.20]{lam}.

We will now use the Witt invariant to find necessary and sufficient conditions for a quadratic form to be an $I^3$-neighbor of some complementary dimension, repeatedly using \cite[Proposition~V.3.20]{lam} and basic properties of quaternion algebras (see, e.g., \cite[Chapter~III, Sections 1 and 2]{lam}) in the process.

\begin{prop}
\label{I^3 neighbors of comp dim 1}
Over a field $k$ of characteristic $\ne 2$, an odd-dimensional quadratic form $q$ of dimension $\geq 3$ is an $I^3$-neighbor of complementary dimension 1 if and only if $q$ has trivial Witt invariant, i.e., $c(q) = 1$.
\end{prop}
\begin{proof}
Let $d = \det q$, and first assume that $q$ is an $I^3$-neighbor of complementary dimension 1; i.e., there is a one-dimensional quadratic form $\sigma_1$ over $k$ such that $q \perp \sigma_1 \in I^3(k)$. Calculating determinants, we have
\[
	\sigma_1 \simeq \left\langle (-1)^{\frac{\dim q + 1}{2}} d\right\rangle.
\]
So $c(q \perp \langle \pm d \rangle) = 1$. Considering cases of $\dim q \mod 8$, a straightforward calculation shows that $c(q) = c(q \perp \langle \pm d \rangle) = 1$, which completes the proof of the forward implication. We show this calculation in the case of $\dim q \equiv 1 \mod 8$, and the other cases follow in a similar fashion.

Since $\dim q \equiv 1 \mod 8$, we have $\sigma_1 \simeq \langle -d \rangle$, and
\[
	1 = c(q \perp \langle -d \rangle) = s(q \perp \langle -d \rangle) = s(q) \left(\frac{d, -d}{k}\right) = s(q) = c(q).
\]

Conversely, suppose that $q$ has trivial Witt invariant. Then letting
\[
	q' = q \perp \left\langle (-1)^{\frac{\dim q + 1}{2}} d \right\rangle,
\]
we have $q' \in I^2(k)$. Moreover, the same calculations as those above show that $c(q') = c(q) = 1$, hence $q' \in I^3(k)$. Therefore $q$ is an $I^3$-neighbor of complementary dimension 1.
\end{proof}
\begin{remark}
A seven-dimensional quadratic form $q$ over a field $k$ of characteristic $\ne 2$ is an $I^3$-neighbor of complementary dimension 1 if and only if $q$ is a Pfister neighbor. So by Proposition~\ref{I^3 neighbors of comp dim 1} we conclude that $q$ is a Pfister neighbor if and only if $c(q) = 1$. This agrees with an observation of Knebusch \cite[p.~11]{knebusch}.
\end{remark}

\begin{prop}
\label{I^3 neighbors of comp dim 2}
Over a field $k$ of characteristic $\ne 2$, an even-dimensional quadratic form $q$ of dimension $\geq 4$ is an $I^3$-neighbor of complementary dimension 2 if and only if $c(q) = \left(\frac{a, d_{\pm}q}{k}\right)$ for some $a \in k^{\times}$, where $d_{\pm}q$ is the signed determinant of $q$.
\end{prop}
\begin{proof}
Let $d = \det q$. We first focus on the reverse implication, so suppose there is some $a \in k^{\times}$ such that $c(q) = \left(\frac{a, d_{\pm}q}{k}\right)$. Now let $q' = q \perp \left\langle -a, (-1)^{\frac{\dim q}{2}}ad \right\rangle$. Then $q' \in I^2(k)$ by construction, and considering cases of $\dim q \mod 8$, calculations show that $c(q') = 1$. So $q' \in I^3(k)$, thus $q$ is an $I^3$-neighbor of complementary dimension 2. We show this calculation when $\dim q \equiv 2 \mod 8$, and the other cases follow in a similar fashion.

Because $\dim q \equiv 2 \mod 8$, we have $q' = q \perp \langle -a, -ad \rangle$, hence
\begin{align*}
	c(q') &= c(q \perp \langle -a, -ad \rangle) = s(q \perp \langle -a, -ad \rangle) \left(\frac{-1, -1}{k}\right) \\
	&= s(q)s(\langle -a, -ad \rangle) \left(\frac{d, d}{k}\right)\left(\frac{-1, -1}{k}\right) = s(q)\left(\frac{-a, -ad}{k}\right)\left(\frac{-1, d}{k}\right)\left(\frac{-1, -1}{k}\right) \\
	&= s(q) \left(\frac{-a, -d}{k}\right)\left(\frac{-1, -d}{k}\right) = s(q)\left(\frac{a, -d}{k}\right) = c(q)\left(\frac{a, d_{\pm}q}{k}\right) = 1.
\end{align*}

Conversely, suppose that $q$ is an $I^3$-neighbor of complementary dimension 2; i.e., there exists a two-dimensional form $\sigma_2$ over $k$ such that $q \perp \sigma_2 \in I^3(k)$. In particular, $q \perp \sigma_2 \in I^2(k)$, so calculating determinants, we must have, for some $a \in k^{\times}$,
\[
	\sigma_2 \simeq \left\langle -a, (-1)^{\frac{\dim q}{2}}ad \right\rangle.
\]
Since $c(q \perp \sigma_2) = 1$, a calculation similar to the one in the previous paragraph shows $c(q) = \left(\frac{a, d_{\pm}q}{k}\right)$, as desired.
\end{proof}
\begin{remarks} Let $q$ be an even-dimensional quadratic form over a field $k$ of characteristic $\ne 2$ with $\dim q \geq 4$. 
\begin{enumerate}[label=(\arabic*)]
	\item If $c(q) = 1 = \left(\frac{1, d_{\pm}q}{k}\right)$, then Proposition \ref{I^3 neighbors of comp dim 2} implies that $q$ is an $I^3$-neighbor of complementary dimension 2.
	
	\item If $q \in I^2(k)$, then by Proposition \ref{I^3 neighbors of comp dim 2}, $q$ is an $I^3$-neighbor of complementary dimension 2 if and only if $q \in I^3(k)$ since $d_{\pm} q = 1$ and $\left(\frac{a, 1}{k}\right) = 1$ for any $a \in k^{\times}$.
	
	\item If $\dim q = 4$ with $d = \det q$ and $K = k\left(\sqrt{d}\right)$, then $q$ is isotropic over $k$ if and only if $c(q_K) = 1$ \cite[Remark~V.3.24]{lam}. If $q$ is an $I^3$-neighbor of complementary dimension 2, then by Proposition \ref{I^3 neighbors of comp dim 2}, $c(q) = \left(\frac{a, d}{k}\right)$ for some $a \in k^{\times}$. Therefore $c(q_K) = \left(\frac{a, 1}{K}\right) = 1$, which implies that $q$ is isotropic over $k$. This agrees with Lemma \ref{I^n neighbors of small dimension} since $2^3 > 6$.
	
	\item If $\dim q = 6$, then $q$ is an $I^3$-neighbor of complementary dimension 2 if and only if $q$ is a Pfister neighbor. So if $d = \det q$, then by Proposition \ref{I^3 neighbors of comp dim 2}, $q$ is a Pfister neighbor if and only if $c(q) = \left(\frac{a, -d}{k}\right)$ for some $a \in k^{\times}$. Thus $c\left(q_{k(\sqrt{-d})}\right) = 1$, i.e., $c(q)$ is split by $k\left(\sqrt{-d}\right)$. This agrees with an observation of Knebusch \cite[p.~10]{knebusch}.
\end{enumerate}
\end{remarks}

\begin{prop}
\label{I^3 neighbors of comp dim 3}
Over a field $k$ of characteristic $\ne 2$, an odd-dimensional quadratic form $q$ of dimension $\geq 5$ is an $I^3$-neighbor of complementary dimension 3 if and only if its Witt invariant is Brauer equivalent to a quaternion algebra, i.e., $c(q) = \left(\frac{a, b}{k}\right) \in \Br(k)$ for some $a, b \in k^{\times}$.
\end{prop}
\begin{proof}
First, suppose that $c(q) = \left(\frac{a, b}{k}\right)$ for some $a, b \in k^{\times}$. Since $q$ is odd-dimensional, for any $\alpha \in k^{\times}$, $c(q) = c(\alpha \cdot q)$ \cite[p.~118]{lam}. Moreover, $q$ is an $I^3$-neighbor if and only if $\alpha \cdot q$ is an $I^3$-neighbor. So after scaling, we may assume $\det q = (-1)^{\frac{\dim q + 1}{2}}$.

Now let $q' = q \perp \langle a, b, -ab \rangle \in I^2(k)$, where $a, b \in k^{\times}$ are such that $c(q) = \left(\frac{a, b}{k}\right)$. Again, considering cases depending on $\dim q \mod 8$, straightforward calculations like those seen in the proofs of Propositions \ref{I^3 neighbors of comp dim 1} and \ref{I^3 neighbors of comp dim 2} show that $c(q') = 1$, which proves that $q$ is an $I^3$-neighbor of complementary dimension 3.

Conversely, suppose $q$ is an $I^3$-neighbor of complementary dimension 3; i.e., there is a three-dimensional form $\sigma_3$ over $k$ such that $q \perp \sigma_3 \in I^3(k)$. Let $d = \det q$. Then $(\pm d) \cdot (q \perp \sigma_3) \in I^3(k)$, where we scale by $(-1)^{\frac{\dim q+1}{2}} d$. So we have $\pm d \cdot q \perp \pm d \cdot \sigma_3 \in I^3(k)$. By our choice of scaling, this implies that $\det (\pm d \cdot \sigma_3) = -1$, so there exist $a, b \in k^{\times}$ such that $\pm d \cdot \sigma_3 \simeq \langle a, b, -ab \rangle$. Thus $\pm d \cdot q \perp \langle a, b, -ab \rangle \in I^3(k)$, which implies $c(\pm d \cdot q \perp \langle a, b, -ab \rangle) = 1$. Using that $q$ is odd-dimensional, a calculation similar to those above shows that $c(q) = c(\pm d \cdot q) = \left(\frac{a, b}{k}\right)$.
\end{proof}
\begin{remark}
A five-dimensional quadratic form $q$ is an $I^3$-neighbor of complementary dimension~$3$ if and only if $q$ is a Pfister neighbor. So Proposition \ref{I^3 neighbors of comp dim 3} implies that a five-dimensional quadratic form is a Pfister neighbor if and only if its Witt invariant is Brauer equivalent to a quaternion algebra, recovering an observation of Knebusch \cite[p.~10]{knebusch}.
\end{remark}

\begin{prop}
Let $q$ be an even-dimensional quadratic form of dimension $\geq 6$ over a field~$k$ of characteristic $\ne 2$. Then $q$ is an $I^3$-neighbor of complementary dimension 4 with a complementary form that represents its determinant if and only if the Witt invariant of $q$ is Brauer equivalent to a quaternion algebra, i.e., $c(q) = \left(\frac{a, b}{k}\right) \in \Br(k)$ for some $a, b \in k^{\times}$.
\end{prop}
\begin{proof}
First, suppose that $c(q) = \left(\frac{a, b}{k}\right) \in \Br(k)$ for some $a, b \in k^{\times}$. Let $d = \det q$, and let $q' = q \perp \left\langle (-1)^{\frac{\dim q }{2}} d \right\rangle$. A calculation similar to those above shows that $c(q) = c(q')$. In particular,~$q'$ is an odd-dimensional quadratic form whose Witt invariant is Brauer equivalent to a quaternion algebra. Thus $q'$ is an $I^3$-neighbor of complementary dimension 3 by Proposition \ref{I^3 neighbors of comp dim 3}. So there is some three-dimensional form $\sigma_3$ over $k$ such that $q' \perp \sigma_3 \in I^3(k)$. By calculating determinants, we have $\det \sigma_3 = 1$. From this we see that $q \perp \langle \pm d \rangle \perp \sigma_3 \in I^3(k)$, where $\langle \pm d \rangle \perp \sigma_3$ has determinant~$\pm d$. The form $\langle \pm d \rangle \perp \sigma_3$ is a four-dimensional complementary form of $q$ that represents its determinant by \cite[Theorem~I.2.3]{lam}, proving the reverse implication.

Conversely, suppose that $q$ is an $I^3$-neighbor of complementary dimension 4 with a complementary form $\sigma_4$ that represents its determinant. In particular, $q \perp \sigma_4 \in I^3(k)$, and letting $d = \det q$, by calculating determinants we have $\det \sigma_4 = (-1)^{\frac{\dim q}{2}} d$. By assumption, $\sigma_4$ represents its determinant, so by \cite[Theorem~I.2.3]{lam}, $\sigma_4 \simeq \langle \pm d \rangle \perp \sigma_3$ for some three-dimensional form $\sigma_3$ over $k$. The form $q \perp \langle \pm d \rangle$ is therefore an odd-dimensional $I^3$-neighbor of complementary dimension 3, so by Proposition \ref{I^3 neighbors of comp dim 3} we have $c(q \perp \langle \pm d \rangle) = \left(\frac{a, b}{k}\right)$ for some $a, b \in k^{\times}$. Finally, since $c(q) = c(q \perp \langle \pm d \rangle)$ (with $d, -d$ chosen according to $\dim q \mod 4$), it follows that $c(q) = \left(\frac{a, b}{k}\right)$.
\end{proof}

\subsection{All quadratic forms}
\label{all forms}
In Subsection \ref{counterexamples} we found numerous counterexamples to $\lgp(r, 1)$ for various integers $r$, and in Subsection \ref{I^n-neighbors} we saw that by requiring a quadratic form $q$ to be an $I^n$-neighbor for some $n$, we can find integers $r, s \geq 1$ such that $q$ satisfies $\lgp(r, s)$. In this subsection we will consider all quadratic forms over a field $k$ and ask whether we can find integers $r, s \geq 1$ such that all quadratic forms over $k$ satisfy $\lgp(r, s)$.

\begin{lemma}
\label{increase r and s}
Let $k$ be a field of characteristic $\ne 2$ equipped with a non-empty set $V$ of non-trivial discrete valuations. Suppose there exist integers $r, s \geq 1$ such that all quadratic forms over $k$ satisfy $\emph{LGP}(r, s)$ with respect to $V$. Then for all integers $j \geq 0$, all quadratic forms over $k$ satisfy $\emph{LGP}(r+j, s+j)$ with respect to $V$.
\end{lemma}
\begin{proof}
We prove the claim by induction on $j \geq 0$. The base case of $j = 0$ is trivial since we assumed all quadratic forms over $k$ satisfy $\lgp(r, s)$ with respect to $V$.

Now suppose for some $j \geq 0$ that all quadratic forms over $k$ satisfy $\lgp(r + j, s + j)$ with respect to $V$, and let $q$ be any quadratic form over $k$ such that $i_W(q_v) \geq r + j + 1$ for all $v \in V$. Then in particular, $i_W(q_v) \geq r + j$ for all $v \in V$, so by the induction hypothesis, $i_W(q) \geq s + j \geq 1$ over $k$. We can therefore write $q \simeq \HH \perp q'$ for some form $q'$ over $k$. Now, for all $v \in V$ we have
\[
	i_W(q_v) = i_W\left((\HH \perp q')_v\right) \geq r + j + 1,
\]
which implies that $i_W(q'_v) \geq r + j$ for all $v \in V$. By the induction hypothesis, the form $q'$ satisfies $\lgp(r + j, s + j)$ with respect to $V$, thus $i_W(q'_v) \geq r + j$ for all $v \in V$ implies that $i_W(q') \geq s + j$. This, in turn, implies that $i_W(q) = i_W\left(\HH \perp q'\right) \geq s + j + 1$, proving the claim by induction.
\end{proof}
A statement stronger than the converse of Lemma \ref{increase r and s} holds.
\begin{lemma}
\label{decreasing r and s}
Let $k$ be a field of characteristic $\ne 2$ equipped with a non-empty set $V$ of non-trivial discrete valuations, and let $r, s \geq 1$ be positive integers. If there exists some integer $j \geq 0$ such that all quadratic forms over $k$ satisfy $\emph{LGP}(r+j, s+j)$ with respect to $V$, then all quadratic forms over~$k$ satisfy $\emph{LGP}(r,s)$ with respect to $V$.
\end{lemma}
\begin{proof}
The claim is trivial if $j = 0$, so suppose there is some integer $j \geq 1$ such that all quadratic forms over $k$ satisfy $\lgp(r+j, s+j)$ with respect to $V$. Let $q$ be any quadratic form over $k$ such that $i_W(q_v) \geq r$ for all $v \in V$, and let $q' = q \perp j\HH$. Then $i_W(q'_v) \geq r + j$ for all $v \in V$, and since all quadratic forms over $k$ satisfy $\lgp(r+j, s+j)$ with respect to $V$, we conclude that $i_W(q') \geq s + j$. Because $q' = q \perp j\HH$, this implies that $i_W(q) \geq s$. Hence $q$ satisfies $\lgp(r, s)$ with respect to~$V$.
\end{proof}

Combining Lemmas \ref{increase r and s} and \ref{decreasing r and s} gives the following result.
\begin{theorem} \label{varying r,s together}
Let $k$ be a field of characteristic $\ne 2$ equipped with a non-empty set $V$ of non-trivial discrete valuations. If there are integers $r, s \geq 1$ such that all quadratic forms over~$k$ satisfy $\emph{LGP}(r, s)$ with respect to $V$, then for any integer $j$ (not necessarily positive) such that $r + j, s + j \geq 1$, all quadratic forms over $k$ satisfy $\emph{LGP}(r+j, s+j)$ with respect to $V$. 
\end{theorem}
Theorem \ref{varying r,s together} motivates the definition of the following invariant associated to $k$ and $V$.
\begin{defin}
\label{l-inv}
For a field $k$ of characteristic $\ne 2$ equipped with a non-empty set $V$ of non-trivial discrete valuations, let
\[
	l(k,V) = \inf\{r \geq 1 \mid \text{all quadratic forms over $k$ satisfy $\lgp(r, 1)$ with respect to $V$}\},
\]
where the infimum of the empty set is $\infty$.
\end{defin}
In particular, $l(k,V) = 1$ if and only if all quadratic forms over $k$ satisfy the local-global principle for isotropy with respect to $V$. The invariant $l(k,V)$ can therefore be seen as a measure of the failure of the local-global principle for isotropy of quadratic forms over $k$ with respect to $V$. The remainder of this subsection will be devoted to studying $l(k,V)$ for various $k, V$.

One of the first observations we make is that for fields $k$ with finite $u$-invariant, $u(k)$ can be used to give a natural upper bound on $l(k,V)$ for any non-empty set $V$ of non-trivial discrete valuations on $k$.
\begin{prop}
\label{first upper bound on l}
Let $k$ be a field of characteristic $\ne 2$ with $u(k) < \infty$. If $V$ is any non-empty set of non-trivial discrete valuations on $k$, then $l(k, V) \leq \left\lceil \frac{u(k) + 1}{2} \right\rceil$.
\end{prop}
\begin{proof}
To prove the proposition, we must show that, with respect to $V$, all quadratic forms over~$k$ satisfy $\lgp\left(\left\lceil \frac{u(k) + 1}{2}\right\rceil, 1\right)$. To that end, let $q$ be any quadratic form over $k$ with $i_W(q_v) \geq \left\lceil \frac{u(k) + 1}{2}\right\rceil$ for all $v \in V$. This implies that
\[
	\dim q \geq 2i_W(q_v) \geq 2 \left\lceil \frac{u(k) + 1}{2}\right\rceil \geq u(k) + 1.
\]
So $\dim q > u(k)$, hence $q$ is isotropic over $k$; i.e., $i_W(q) \geq 1$. Thus $q$ satisfies $\lgp\left(\left\lceil \frac{u(k) + 1}{2} \right\rceil, 1\right)$ with respect to $V$.
\end{proof}

The upper bound for $l(k,V)$ given by Proposition \ref{first upper bound on l} is not always sharp, as the following example illustrates.
\begin{example}
\label{sharper upper bound}
Let $k$ be any field of characteristic $\ne 2$ with $u(k) = 2n$ for some integer $n \geq 1$, and suppose $k$ is equipped with a non-empty set $V$ of non-trivial discrete valuations with respect to which the local-global principle for isometry holds (e.g., $k = K(x_1, \ldots, x_r)$ for an integer $r \geq 2$ and $K$ an algebraically closed field of characteristic $\ne 2$ equipped with the set $V = V_r$ of non-trivial discrete valuations on $k$ that are trivial on $K(x_1, \ldots, x_{r-1})$). Then
\[
	l(k, V) \leq n < \left\lceil \frac{u(k) + 1}{2}\right\rceil = n + 1.
\]
Indeed, to show that $l(k, V) \leq n$, we must show that all quadratic forms over $k$ satisfy $\lgp(n, 1)$ with respect to $V$. To that end, let $q$ be any quadratic form over $k$ such that $i_W(q_v) \geq n$ for all $v \in V$. Then $\dim q \geq 2i_W(q_v) \geq 2n = u(k)$. If $\dim q > 2n$, then $q$ is automatically isotropic over~$k$, and therefore satisfies $\lgp(n, 1)$ with respect to $V$. If $\dim q = 2n$ and $i_W(q_v) \geq n$ for all $v \in V$, then $i_W(q_v) = n$ for all $v \in V$, and $q_v$ is hyperbolic over $k_v$ for all $v \in V$. Because the local-global principle for isometry holds over $k$ with respect to $V$, this implies that $q$ is hyperbolic over $k$, i.e., $i_W(q) = n \geq 1$. Thus $q$ satisfies $\lgp(n, 1)$ with respect to $V$.
\end{example}

Knowing that counterexamples to $\lgp(r, 1)$ exist over a field $k$ with respect to $V$ also allows us to give lower bounds on $l(k, V)$.
\begin{example}
Let $\ell$ be a field of characteristic $\ne 2$ such that $\ell \in \mathscr{A}_i(2)$ for some $i \geq 0$ and $u(\ell) = 2^i$. For any integer $r \geq 1$ let $L_r = \ell(x_1, \ldots, x_r)$, and for $r \geq 2$ let $V_r$ be the set of non-trivial discrete valuations on $L_r$ that are trivial on $L_{r-1}$. Then by Theorem \ref{ce to lgp}, for $r \geq 2$ there exists a quadratic form over $L_r$ that violates $\lgp\left(2^{i+r-2}, 1\right)$ with respect to $V_r$. Thus $l(L_r, V_r) \geq 2^{i+r - 2} + 1$. Furthermore, by the same reasoning as that in Example \ref{sharper upper bound}, we know that $l(L_r, V_r) \leq 2^{i+r-1}$. Therefore $2^{i+r-2} + 1 \leq l(L_r, V_r) \leq 2^{i + r - 1}.$
\end{example}
We conclude this section by computing $l(F, V)$ for $V$ the set of all discrete valuations on a semi-global field $F$, i.e., the function field of a curve over a complete discretely valued field $K$. Recall that a \textit{regular model} of $F$ is a regular connected projective curve $\mathscr{X}$ over the valuation ring~ $T$ of~$K$ with function field $F$. To any regular model $\mathscr{X}$ of $F$ we can associate a combinatorial object called the \textit{reduction graph} (defined in \cite[Section 6]{hhk15}), and we will use the reduction graph to compute $l(F,V)$. The reduction graph encodes the intersection configuration of the irreducible components of the closed fiber of a regular model of $F$, and Harbater, Hartmann, and Krashen have shown that, for various algebraic objects over $F$, the validity of certain local-global principles depends on whether the reduction graph is a tree. A regular model of $F$ always exists, but this model is not unique. However, whether the reduction graph is a tree does not depend on the choice of model \cite[Corollary~7.8]{hhk15}.
\begin{prop}
\label{l of a sg field}
Let $T$ be a complete discrete valuation ring with fraction field $K$ and residue field $k$ of characteristic $\ne 2$. Let $F$ be a one-variable function field over $K$, and let $V$ be the set of all discrete valuations on $F$. Then $l(F, V) \leq 2$, with $l(F, V) = 2$ if and only if the reduction graph of a regular model of $F$ is not a tree.
\end{prop}
\begin{proof}
We first show that $l(F, V) \leq 2$. To prove this claim, we must show that all quadratic forms over $F$ satisfy $\lgp(2, 1)$ with respect to $V$. To that end, let $q$ be any quadratic form over~$F$ such that $i_W(q_v) \geq 2$ for all $v \in V$. This assumption on the Witt index of $q_v$ implies that $\dim q \geq 4$. Moreover, since $i_W(q_v) \geq 2$ for all $v \in V$, $q$ is isotropic over $F_v$ for all $v \in V$. Because $\dim q \geq 3$ and $q$ is isotropic over $F_v$ for all $v \in V$, by \cite[Theorem~3.1]{cps12} we conclude that $q$ is isotropic over $F$, i.e, $i_W(q) \geq 1$. Therefore $q$ satisfies $\lgp(2, 1)$ with respect to $V$.

From the above paragraph, we see that $l(F,V) = 2$ if and only if $l(F, V) \ne 1$, i.e., there is a quadratic form over $F$ that violates $\lgp(1,1)$ with respect to $V$. By \cite[Theorem~3.1]{cps12}, all quadratic forms over $F$ of dimension $\geq 3$ satisfy $\lgp(1,1)$ with respect to $V$. So $l(F,V) = 2$ if and only if there is a two-dimensional quadratic form over $F$ that violates $\lgp(1,1)$ with respect to $V$. By \cite[Theorem~9.11]{hhk15}, such a two-dimensional counterexample exists if and only if the reduction graph of a regular model of $F$ is not a tree.
\end{proof}

\section{Refining the $m$-invariant}
\label{refined m-inv}
\subsection{Definition and general results}
\label{refined m generalities}
Let $k$ be a field of characteristic $\ne 2$. The \textit{$m$-invariant} of~$k$, denoted by $m(k)$, is defined to be the minimal dimension of an anisotropic universal quadratic form over~$k$ \cite{m-inv}. Using the Witt index and the First Representation Theorem \cite[Corollary~I.3.5]{lam}, we can rewrite $m(k)$ as
\[
	m(k) = \inf\left\{\dim q \geq 1 \middle\vert \begin{array}{c} q \text{ a quadratic form over $k$ such that } i_W(q) < 1 \text{ and} \\
	i_W(q \perp \sigma_1) \geq 1 \text{ for all one-dimensional quadratic forms $\sigma_1$ over $k$} \end{array}\right\}.
\]
This motivates the following refined version of the $m$-invariant of $k$.
\begin{defin}
\label{refined m}
Let $k$ be a field of characteristic $\ne 2$. For positive integers $i, j \geq 1$, let
\[
	m_{i,j}(k) = \inf\left\{\dim q \geq 1 \middle \vert \begin{array}{c} q \text{ a quadratic form over $k$ such that } i_W(q) < j \text{ and} \\
	i_W(q \perp \sigma_i) \geq j \text{ for all $i$-dimensional quadratic forms $\sigma_i$ over $k$}\end{array} \right\},
\]
where the infimum of the empty set is $\infty$.
\end{defin}
In particular, $m_{1,1}(k) = m(k)$. Next, we define terminology that we will use throughout Section~\ref{refined m-inv}.

\begin{defin}
    \label{(i,j)-realizing form}
Let $k$ be a field of characteristic $\ne 2$, and let $i,j \geq 1$ be any positive integers. We call a regular quadratic form $q$ of dimension $\geq 1$ over $k$ an $(i,j)$-\textit{realizing form} if $i_W(q) < j$ and $i_W(q \perp \sigma_i) \geq j$ for all $i$-dimensional quadratic forms $\sigma_i$ over $k$.
\end{defin}
For example, any anisotropic universal quadratic form over $k$ is a $(1,1)$-realizing form. Furthermore, with this definition we can rephrase $m_{i,j}(k)$ as the minimal dimension of an $(i,j)$-realizing form over $k$. In particular, $m_{i,j}(k) < \infty$ if and only if an $(i,j)$-realizing form exists over $k$.

We will begin this subsection by proving some initial results about how the invariants $m_{i,j}(k)$ relate to one another as we increase $i$ or $j$. We will then show, despite the fact that finding precise values of $m_{i,j}(k)$ is a challenging problem in general, that there are natural upper and lower bounds for~$m_{i,j}(k)$ given in terms of $u(k)$ and $m(k)$ (see, e.g., Theorem \ref{upper and lower bounds}). While finding these lower bounds, we will also prove sharper results about how the invariants $m_{i,j}(k)$ relate to one another as we decrease $i$ or $j$ (see, e.g., Proposition \ref{decreasing i in m_{i,j}} and Lemma \ref{decreasing j in m_{1,j}}).

\begin{lemma}
\label{initial inequalities}
Let $k$ be any field of characteristic $\ne 2$, and let $i, j \geq 1$ be any positive integers.
\begin{enumerate}[label=(\alph*)]
	\item For any integer $i' \geq i$, if $m_{i,j}(k) < \infty$, then $m_{i', j}(k) \leq m_{i,j}(k)$.
	
	\item For any integer $j' \geq j$, if $m_{i,j'}(k) < \infty$, then $m_{i, j'}(k) \geq m_{i,j}(k)$.
\end{enumerate}
\end{lemma}
\begin{proof}
(a) Let $q$ be an $m_{i,j}(k)$-dimensional $(i,j)$-realizing form over $k$. Then letting $\sigma_{i'}$ be any $i'$-dimensional quadratic form over $k$ and taking any $i$-dimensional subform $\sigma_i$ of $\sigma_{i'}$, we have $i_W\left(q \perp \sigma_{i'}\right) \geq i_W\left(q \perp \sigma_i\right) \geq j$. Since $i_W(q) < j$, we conclude $m_{i',j}(k) \leq \dim q = m_{i,j}(k)$.
	
(b) If $j' = j$, then there is nothing to prove, so suppose $j' = j + s$ for some $s > 0$. Let~$q$ be an $m_{i,j'}(k)$-dimensional $(i,j')$-realizing form over $k$. If $i_W(q) < j$, then we are done since $m_{i,j'}(k) = \dim q \geq m_{i,j}(k)$ by the definition of $m_{i,j}(k)$. Otherwise, $j \leq i_W(q) < j'$ and we have $i_W(q) = j + \ell \geq \ell + 1$ for some $0 \leq \ell \leq s - 1$. So we can write $q \simeq \left(\ell + 1\right) \HH \perp q'$ for some form $q'$ over~$k$ with $i_W(q') = j - 1 < j$. Moreover, for all $i$-dimensional forms $\sigma_i$ over $k$ we have
	\[
		i_W(q \perp \sigma_i) = i_W\left((\ell+1)\HH \perp q' \perp \sigma_i\right) \geq j' = j + s.
	\]
	This implies $i_W\left(q' \perp \sigma_i\right) \geq j + s - (\ell + 1) \geq j$, where this last inequality holds because $\ell \leq s - 1$. Therefore $q'$ is an $(i,j)$-realizing form over $k$, hence $m_{i,j}(k) \leq \dim q' \leq \dim q = m_{i,j'}(k)$.
\end{proof}
We will now prove that a natural upper bound exists for $m_{i,j}(k)$ in terms of $u(k)$, $i$, and $j$.
\begin{lemma}
\label{upper bound on m_{i,j}}
Let $k$ be a field of characteristic $\ne 2$ with $u(k) < \infty$, and let $i,j \geq 1$ be any positive integers. Then $m_{i,j}(k) \leq \max\{1, u(k) + 2j - 1 - i\}$.
\end{lemma}
\begin{proof}
If $1 \geq u(k) + 2j - 1 - i$, then we need to show that $m_{i,j}(k) \leq 1$. Take any $a \in k^{\times}$. Then $i_W(\langle a \rangle) = 0 < j$, and for all $i$-dimensional forms $\sigma_i$ over $k$ we have
\[
	\dim \left(\langle a \rangle \perp \sigma_i \right) = i + 1 \geq u(k) + 2j - 1.
\]
Therefore $i_W(\langle a \rangle \perp \sigma_i) \geq j$ by Lemma \ref{Witt index of large dimensional forms}. This shows that $m_{i,j}(k) \leq 1$, as desired.

If $1 \leq u(k) + 2j - 1 - i \leq u(k)$, then let $q$ be any anisotropic quadratic form over $k$ of dimension $u(k) + 2j - 1 - i$. Such a form $q$ exists by the definition of $u(k)$. Then $i_W(q) = 0 < j$, and for all $i$-dimensional forms $\sigma_i$ over $k$ we have $\dim (q \perp \sigma_i) = u(k) + 2j - 1$. This implies $i_W(q \perp \sigma_i) \geq j$ by Lemma \ref{Witt index of large dimensional forms}. Thus $m_{i,j}(k) \leq \dim q = u(k) + 2j - 1 - i$.

Finally, suppose $u(k) + 2j - 1 - i > u(k)$, i.e., $2j - 1 > i$. Let $q$ be any anisotropic quadratic form over $k$ of dimension $u(k)$. Define the $(u(k) + 2j - 1 - i)$-dimensional quadratic form $\varphi_{i,j}$ over~$k$ as follows:
\[
	\varphi_{i,j} = \begin{cases}
		q \perp \frac{2j - 1 - i}{2} \HH &\text{ if $i$ is odd}, \\
		q \perp \langle 1 \rangle \perp \frac{2j - 2 - i}{2} \HH &\text{ if $i$ is even}.
	\end{cases}
\]
Then $i_W(\varphi_{i,j}) = j - \lceil \frac{i}{2} \rceil < j$. Furthermore, $\dim(\varphi_{i,j} \perp \sigma_i) = u(k) + 2j - 1$ for all $i$-dimensional forms~$\sigma_i$ over~$k$, hence $i_W\left(\varphi_{i,j} \perp \sigma_i\right) \geq j$ by Lemma \ref{Witt index of large dimensional forms}. Thus $m_{i,j}(k) \leq \dim \varphi_{i,j} = u(k) + 2j - 1 - i$.

We have found an $(i,j)$-realizing form over $k$ whose dimension is $\max\{1, u(k) + 2j - 1 - i\}$ in all cases for $i$ and $j$, so the proof is complete.
\end{proof}

We now shift our focus to finding a similar lower bound for $m_{i,j}(k)$ in terms of $m(k)$, $i$, and $j$.

\begin{lemma}
\label{add on a form}
Let $k$ be any field of characteristic $\ne 2$, and let $i, j, r$ be integers such that $i \geq 2$, $j \geq 1$, and $1 \leq r < i$. Let $q$ be a quadratic form over $k$ such that $i_W(q) < j$ and $\dim q \leq m_{i-r, j}(k) - r$. Then there exists an $r$-dimensional quadratic form $\sigma_r$ over $k$ such that $i_W(q \perp \sigma_r) < j$.
\end{lemma}
\begin{proof}
Since $i_W(q) < j$ and $\dim q < m_{i-r,j}(k)$, there exists an $(i-r)$-dimensional form $\sigma_{i-r}$ over~$k$ such that $i_W(q \perp \sigma_{i-r}) < j$. Because $i - r \geq 1$, we can take any entry $a_1$ of $\sigma_{i-r}$ and have $i_W(q \perp \langle a_1 \rangle) < j$. If $r = 1$, then we are done, letting $\sigma_1 = \langle a_1 \rangle$.

If $r \geq 2$, then consider the form $q \perp \langle a_1 \rangle$ over $k$. We have
\[
	i_W(q \perp \langle a_1 \rangle) < j \text{ and } \dim (q \perp \langle a_1 \rangle) \leq m_{i-r,j}(k) - r + 1 < m_{i-r, j}(k).
\]
By the same reasoning as in the previous paragraph, we can find an element $a_2 \in k^{\times}$ such that $i_W(q \perp \langle a_1, a_2 \rangle) < j$. By repeating this process, we can find $r$ elements $a_1, \ldots, a_r \in k^{\times}$ such that $i_W(q \perp \langle a_1, \ldots, a_r \rangle) < j$. Letting $\sigma_r = \langle a_1, \ldots, a_r \rangle$ completes the proof.
\end{proof}
\begin{prop}
\label{decreasing i in m_{i,j}}
Let $k$ be any field of characteristic $\ne 2$, and let $i, j \geq 1$ be any positive integers. For any integer $r$ such that $0 \leq r < i$, if $m_{i,j}(k) < \infty$, then $m_{i,j}(k) \geq m_{i-r,j}(k) - r$.
\end{prop}
\begin{proof}
This is trivial for $r = 0$ since both sides equal $m_{i,j}(k)$, so we may assume $r \geq 1$, which implies that $i \geq 2$. By contradiction, assume we can find integers $i,j,r$ with $i \geq 2$, $j \geq 1$, and $1 \leq r < i$ such that $m_{i,j}(k) \leq m_{i-r,j}(k) - r - 1$. Let $q$ be an $m_{i,j}(k)$-dimensional $(i,j)$-realizing form over~$k$. Since $i_W(q) < j$ and $\dim q \leq m_{i-r,j}(k) - r$, by Lemma~\ref{add on a form} there is an $r$-dimensional quadratic form $\sigma_r$ over $k$ such that $i_W(q \perp \sigma_r) < j$. Now, because $\dim q = m_{i,j}(k) \leq m_{i-r,j}(k) - r - 1$, we have $\dim (q \perp \sigma_r) \leq m_{i-r,j}(k) - 1 < m_{i-r,j}(k)$. Moreover, since $i_W(q \perp \sigma_r) < j$, this implies that there exists an $(i-r)$-dimensional form $\sigma_{i-r}$ over $k$ such that $i_W(q \perp \sigma_r \perp \sigma_{i-r}) < j$. But $\dim (\sigma_r \perp \sigma_{i-r}) = i$, so this is a contradiction of our choice of $q$, and the proof is complete.
\end{proof}
Letting $r = i - 1$ in Proposition \ref{decreasing i in m_{i,j}}, we have
\begin{cor}
\label{decreasing i to 1}
Let $k$ be any field of characteristic $\ne 2$, and let $i,j \geq 1$ be any positive integers. If $m_{i,j}(k) < \infty$, then $m_{i,j}(k) \geq m_{1, j}(k) - i + 1$.
\end{cor}

Proposition \ref{decreasing i in m_{i,j}} shows how much $m_{i,j}(k)$ can increase by decreasing $i$, and the next lemma shows how much $m_{1,j}(k)$ can change when we decrease $j$.
\begin{lemma}
\label{decreasing j in m_{1,j}}
Let $k$ be any field of characteristic $\ne 2$, and let $j \geq 1$ be any positive integer. If $m_{1,j}(k) < \infty$, then $m_{1,j}(k) \geq m_{1,1}(k) + 2j - 2$.
\end{lemma}
\begin{proof}
Let $q$ be an $m_{1,j}(k)$-dimensional $(1,j)$-realizing form over $k$. Then for all 1-dimensional forms~$\sigma_1$ over $k$ we have $i_W(q \perp \sigma_1) \geq j$. Lemma \ref{Witt index of sum} then implies $1 + i_W(q) \geq i_W(q \perp \sigma_1) \geq j$, hence $i_W(q) \geq j - 1$. By assumption, $i_W(q) < j$, so $i_W(q) = j - 1$ and we can write $q \simeq (j-1)\HH \perp q_{an}$, where $q_{an}$ is anisotropic over $k$. So for all 1-dimensional forms $\sigma_1$ over $k$,
\[
	i_W(q \perp \sigma_1) = i_W\left((j-1)\HH \perp q_{an} \perp \sigma_1\right) \geq j.
\]
This implies that $i_W(q_{an} \perp \sigma_1) \geq 1$. Therefore $q_{an}$ is anisotropic and universal over $k$, hence $\dim q_{an} \geq m(k) = m_{1,1}(k)$. Thus $m_{1,j}(k) = \dim q = 2j - 2 + \dim q_{an} \geq 2j - 2 + m_{1,1}(k)$.
\end{proof}

We can now write a lower bound for $m_{i,j}(k)$ in terms of $m(k)$, $i$, and $j$.
\begin{cor}
\label{lower bound}
Let $k$ be any field of characteristic $\ne 2$, and let $i, j \geq 1$ be any positive integers. If $m_{i,j}(k) < \infty$, then $m_{i,j}(k) \geq \max\{1, m(k) + 2j - 1 - i\}$.
\end{cor}
\begin{proof}
If $m(k) + 2j - 1 - i < 1$, then the claim is automatic by definition since $m_{i,j}(k) \geq 1$. So suppose $m(k) + 2j - 1 - i \geq 1$, and recall that $m(k) = m_{1,1}(k)$. Then applying Corollary \ref{decreasing i to 1} and Lemma \ref{decreasing j in m_{1,j}}, we have
\[
	m_{i,j}(k) \geq m_{1,j}(k) - i + 1 \geq m_{1,1}(k) + 2j - 2 - i + 1 = m(k) + 2j - 1 - i.
\]
\end{proof}
\begin{theorem}
\label{upper and lower bounds}
Let $k$ be a field of characteristic $\ne 2$ with $u(k) < \infty$, and let $i,j \geq 1$ be any positive integers. Then
\[
	\max\{1, m(k) + 2j - 1 - i\} \leq m_{i,j}(k) \leq \max\{1, u(k) + 2j - 1 - i \}.
\]
Moreover, if either $u(k) + 2j - 1 - i \leq 1$ or $m(k) = u(k)$, then both inequalities above are equalities.
\end{theorem}
\begin{proof}
The desired inequalities are shown in Lemma \ref{upper bound on m_{i,j}} and Corollary \ref{lower bound}.

To prove the second claim, we first observe that $m(k) + 2j - 1 - i \leq u(k) + 2j - 1 - i$ since $m(k) \leq u(k)$. So if either $u(k) + 2j - 1 - i \leq 1$ or $u(k) = m(k)$, then the first and third quantities above are equal, thus the inequalities become equalities.
\end{proof}

\begin{remark}
\label{completely determined}
Let $k$ be a field of characteristic $\ne 2$ such that $m(k) = u(k) < \infty$ (e.g., $k =$~$\C, \C(x)$, $\F_p, \F_p(x)$, $\F_p((t_1)) \cdots ((t_n)), \F_p((t_1)) \cdots ((t_n))(x)$, $\Q_p$, or $\Q_p(x)$ for $p \ne 2$ (see \cite{universal, m-inv})). Then by Theorem \ref{upper and lower bounds}, for any integers $i,j \geq 1$ such that $m(k) + 2j - 1 - i \geq 1$, we have $m_{i,j}(k) = m(k) + 2j - 1 - i$. This can be used to give precise expressions detailing how the invariants $m_{i,j}(k)$ change as we vary~$i$ and~$j$. Indeed, if $0 \leq r < i$, then
\[
	m_{i-r,j}(k) = m(k) + 2j - 1 - (i-r) = m(k) + 2j - 1 - i + r = m_{i,j}(k) + r.
\]
Moreover, for any $s \geq 0$,
\[
	m_{i,j+s}(k) = m(k) + 2(j+s) - 1 - i = m(k) + 2j - 1 - i + 2s = m_{i,j}(k) + 2s.
\]
In other words, if $m(k) = u(k) < \infty$, then the invariants $m_{i,j}(k)$ are completely determined by~$m(k)$,~$i$, and $j$.
\end{remark}

\begin{question}
\label{what if}
How closely do the invariants $m_{i,j}(k)$ follow the behavior seen in Remark \ref{completely determined} as we vary $i$ and $j$ if $m(k) < u(k)$?
\end{question}
We will first focus on answering Question \ref{what if} as we vary $i$. By Proposition \ref{decreasing i in m_{i,j}}, for any field $k$ of characteristic $\ne 2$ and integers $i, j, r$ such that $i \geq 2$, $j \geq 1$, and $1 \leq r < i$, if $m_{i,j}(k) < \infty$, then
\[
	m_{i,j}(k) \geq m_{i-r, j}(k) - r.
\]
However, as we will see, the only inequality holding in the opposite direction for a general field $k$ is that of Lemma~\ref{initial inequalities}(a), i.e., 
\[
	m_{i,j}(k) \leq m_{i-r,j}(k).
\]
Before seeing an example (Example \ref{refined m for sg field with loop}) that illustrates this point, we prove a lemma.
\begin{lemma}
\label{lower bound on m_{i,1} for small i}
Let $k$ be a field of characteristic $\ne 2$ with $2 \leq u(k) < \infty$. If $i$ is any integer such that $1 \leq i \leq u(k) - 1$, then $m_{i,1}(k) \geq 2$.
\end{lemma}
\begin{proof}
By Lemma \ref{initial inequalities}(a), we know that $m_{i', 1}(k) \leq m_{i,1}(k)$ for any $i' \geq i$. Therefore, to prove the lemma, it suffices to show that $m_{u(k) - 1, 1}(k) \geq 2$, which is equivalent to showing $m_{u(k) - 1, 1}(k) \ne 1$.

By contradiction, assume $m_{u(k) - 1, 1}(k) = 1$. That is, assume there is some element $a \in k^{\times}$ such that $i_W(\langle a \rangle \perp \sigma_{u(k) - 1}) \geq 1$ for all $(u(k) - 1)$-dimensional forms $\sigma_{u(k) - 1}$ over $k$. Let $q$ be any anisotropic quadratic form over $k$ with $\dim q = u(k)$. Then because $\dim q = u(k)$, the form $q$ is universal over $k$, hence represents $a$. We can therefore write $q \simeq \langle a \rangle \perp q'$ for some form $q'$ over $k$ with $\dim q' = u(k) - 1$ \cite[Theorem~I.2.3]{lam}. This, however, contradicts our choice of $a$ since $i_W(\langle a \rangle \perp q') = i_W(q) = 0 < 1$. So no such $a$ can exist, completing the proof.
\end{proof}

\begin{example}
\label{refined m for sg field with loop}
Let $p \ne 2$ be a prime, and let $F$ be a semi-global field over $\Q_p$ with a regular model whose reduction graph is not a tree (see the paragraph preceding Proposition \ref{l of a sg field} for terminology). Then we know that $m_{1,1}(F) = m(F) = 2$ \cite[Lemma~3.6]{universal}, and $u(F) = 8$ \cite[Corollary~4.15]{hhk09}. So for any integer $i$ such that $1 \leq i \leq 7$ we have
\[
	m_{i,1}(F) = m_{1,1}(F) = 2.
\]
Indeed, $m_{i,1}(F) \leq m_{1,1}(F) = 2$ by Lemma \ref{initial inequalities}(a), and $m_{i,1}(F) \geq 2$ by Lemma \ref{lower bound on m_{i,1} for small i}.
\end{example}
\begin{remark}
Let $F$ be a field as in Example \ref{refined m for sg field with loop}. Then $2 = m(F) < u(F) = 8$. Furthermore, for such a field $F$ and for $j = 1$ and $2 \leq i \leq 6$, both inequalities of Theorem \ref{upper and lower bounds} are strict. Indeed, $\max\{1, u(F) + 2j - 1 - i\} = 9 - i > 2$,
and since $m(F) = 2$ we have $\max\{1, m(F) + 2j - 1 - i\} = 1$.
We showed in Example \ref{refined m for sg field with loop} that $m_{i,1}(F) = 2$ for $1 \leq i \leq 7$. So for $j = 1$ and $2\leq i \leq 6$ we have
\[
	\max\{1, m(F) + 2j - 1 - i\} < m_{i,j}(F) < \max\{1, u(F) + 2j - 1 - i\}.
\]
\end{remark}

Lemma \ref{lower bound on m_{i,1} for small i} also allows us to show that two fields with the same refined $m$-invariants must have the same $u$-invariant.
\begin{theorem}
\label{refined m determine u}
Let $k,k'$ be fields of characteristic $\ne 2$. If $m_{i,1}(k) =$~$m_{i,1}(k')$ for all integers $i \geq 1$, then $u(k) = u(k')$.
\end{theorem}
\begin{proof}
First, assume that $u(k) = \infty$ and, by contradiction, assume $u(k') < \infty$. Let $i = u(k')$. Then $m_{i,1}(k') = 1$ by Lemma \ref{upper bound on m_{i,j}}, so $m_{i,1}(k) = 1$. Therefore there is some $a \in k^{\times}$ such that $\langle a \rangle \perp \sigma_i$ is isotropic over $k$ for all $i$-dimensional forms $\sigma_i$ over~$k$. This implies that all $i$-dimensional quadratic forms over $k$ represent $-a$ \cite[Corollary~I.3.5]{lam}. Now let $\sigma_{2i}$ be any $2i$-dimensional quadratic form over $k$. Then we can write $\sigma_{2i} \simeq \sigma_i \perp -\sigma_i'$ for some $i$-dimensional forms $\sigma_i, \sigma_i'$ over $k$. Both~$\sigma_i$ and $\sigma_i'$ represent $-a$, so $\sigma_{2i}$ is isotropic by \cite[Corollary~I.3.6]{lam}. Since $\sigma_{2i}$ was arbitrary, we conclude $u(k) < 2i < \infty$. This contradicts our assumption that $u(k) = \infty$, hence $u(k') = \infty$.

Now assume $u(k), u(k') < \infty$. By contradiction, suppose $u(k) \ne u(k')$, and without loss of generality, assume $u(k) <$~$u(k')$. In particular, since $u(k) \geq 1$, we conclude that $u(k') \geq 2$. Now, let $i = u(k') - 1$. Then $i \geq u(k)$, so  $m_{i,1}(k) = 1$ by Lemma \ref{upper bound on m_{i,j}}. On the other hand, since $2 \leq u(k') < \infty$, Lemma \ref{lower bound on m_{i,1} for small i} implies that $m_{i,1}(k') \geq 2$. Thus $m_{i,1}(k) = 1 < 2 \leq m_{i,1}(k')$. This contradicts our assumption that $m_{i,1}(k) = m_{i,1}(k')$, so $u(k) = u(k')$, as desired.
\end{proof}

\begin{remark}
\label{do the m_{i,1} determine m_{i,j}}
Theorem \ref{refined m determine u} shows that the invariants $m_{i,1}$ determine the $u$-invariant of a field. It would also be interesting to know if the values of the invariants $m_{i,1}$ together determine the values of all the invariants $m_{i,j}$. More precisely, if two fields $k, k'$ of characteristic $\ne 2$ satisfy $m_{i,1}(k) = m_{i,1}(k')$ for all integers $i \geq 1$, then must $m_{i,j}(k) = m_{i,j}(k')$ for all integers $i,j \geq 1$?
\end{remark}

We now focus on answering Question \ref{what if} as we let $j$ vary. As we saw in Remark \ref{completely determined}, if $m(k) = u(k) < \infty$, then for any integers $i, j \geq 1$ such that $m(k) + 2j - 1 - i \geq 1$ we have
\[
	m_{i,j+s}(k) = m_{i,j}(k) + 2s
\]
for all integers $s \geq 0$. The next few results show that, even if $m(k) < u(k)$, the invariants $m_{i,j}(k)$ follow the behavior of Remark \ref{completely determined} more closely as we vary $j$ than they do as we vary $i$ (compare Example \ref{refined m for sg field with loop} to Corollary \ref{exact values for m_{i,j} for large j}).

\begin{lemma}
\label{increasing j in m_{i,j}}
Let $k$ be any field of characteristic $\ne 2$, and let $i, j, s \geq 1$ be any positive integers. If $m_{i,j}(k) < \infty$, then $m_{i, j+s}(k) \leq m_{i,j}(k) + 2s$.
\end{lemma}
\begin{proof}
Let $q$ be an $m_{i,j}(k)$-dimensional $(i,j)$-realizing form over $k$. Then the form $q \perp s\HH$ satisfies $i_W(q \perp s\HH) = i_W(q) + s < j + s$, and for all $i$-dimensional forms $\sigma_i$ over $k$ we have
\[
	i_W(q \perp s\HH \perp \sigma_i) = s + i_W(q \perp \sigma_i) \geq j + s.
\]
By definition, $m_{i,j+s}(k) \leq \dim (q \perp s\HH) = m_{i,j}(k) + 2s$.
\end{proof}

In the opposite direction, for particular values of $i$ relative to $u(k)$ and $j$, we can find an inequality that is sharper than the one given in Lemma \ref{initial inequalities}(b) (see Proposition \ref{m_{i,j} increases by 1}). First, we prove a lemma.
\begin{lemma}
\label{m_{i,j} at least 2}
    Let $k$ be a field of characteristic $\ne 2$ with $u(k) < \infty$, and let $i, j \geq 1$ be positive integers such that $i < u(k) + 2j - 2$. Then $m_{i,j}(k) \geq 2$.
\end{lemma}
\begin{proof}
    First, assume $i < u(k)$ (which implies that $u(k) \geq 2$). Then for any $j \geq 1$, by Lemma \ref{initial inequalities}(b) we have $m_{i,j}(k) \geq m_{i,1}(k) \geq 2$, where this last inequality follows from Lemma \ref{lower bound on m_{i,1} for small i}. This completes the proof if $i < u(k)$.

    Next, by contradiction, assume there are integers $i, j \geq 1$ such that $u(k) \leq i < u(k) + 2j - 2$ and $m_{i,j}(k) = 1$. This condition on $i$ implies that $j \geq 2$, and since $u(k) \leq i < u(k) + 2j - 2$, we can write $i = u(k) + s$ for some $0 \leq s \leq 2j - 3$. Next, because $m_{i,j}(k) = 1$, there must be some $a \in k^{\times}$ such that $i_W(\langle a \rangle \perp \sigma_i) \geq j$ for all $i$-dimensional forms $\sigma_i$ over $k$. From this inequality and Lemma \ref{Witt index of sum} we conclude that $i_W(\sigma_i) \geq j - 1$ for all $i$-dimensional forms $\sigma_i$ over~$k$.  Now let $q$ be any anisotropic $u(k)$-dimensional quadratic form over $k$, and let $\sigma_i$ be the $i$-dimensional quadratic form over $k$ defined by
\[
	\sigma_i = \begin{cases}
		q \perp \frac{s}{2}\HH &\text{ if $s$ is even,} \\
		q \perp \langle -a \rangle \perp \frac{s -1}{2}\HH &\text{ if $s$ is odd.}
	\end{cases}
\]
If $s$ is even, then $i_W(\sigma_i) = \frac{s}{2} \leq j - 2 < j - 1$, which is a contradiction. If $s$ is odd, then
\[
	i_W(\langle a \rangle \perp \sigma_i) = i_W\left(q \perp \frac{s+1}{2}\HH\right) = \frac{s+1}{2} \leq j - 1 < j.
\]
This is also a contradiction, so no such $a \in k^{\times}$ exists.
\end{proof}

\begin{prop}
\label{m_{i,j} increases by 1}
Let $k$ be a field of characteristic $\ne 2$ with $u(k) < \infty$. Let $j \geq 1$ be any positive integer, and let $i$ be a positive integer such that $1 \leq i < u(k) + 2j$. Then $m_{i,j+1}(k) \geq m_{i,j}(k) + 1$.
\end{prop}
\begin{proof}
By assumption, $1 \leq i < u(k) + 2(j+1)- 2$. Thus $m_{i, j+1}(k) \geq 2$ by Lemma \ref{m_{i,j} at least 2}. Moreover, since $u(k) < \infty$, Lemma \ref{upper bound on m_{i,j}} implies that $m_{i,j+1}(k) < \infty$. Let $q$ be a quadratic form over $k$ of dimension $m_{i,j+1}(k)$ such that $i_W(q) < j + 1$ and $i_W(q \perp \sigma_i) \geq j + 1$ for all $i$-dimensional forms~$\sigma_i$ over $k$.

If $q$ is isotropic, then we can write $q \simeq \HH \perp q'$ for some form $q'$ over $k$ with $i_W(q') < j$ and $i_W(q' \perp \sigma_i) \geq j$ for all $i$-dimensional forms $\sigma_i$ over $k$. Therefore $\dim q' \geq m_{i,j}(k)$, so
\[
	m_{i,j+1}(k) = \dim q = \dim q' + 2 \geq m_{i,j}(k) + 2 > m_{i,j}(k) + 1.
\]
If $q$ is anisotropic, then writing $q \simeq \langle a \rangle \perp q'$ for some $a \in k^{\times}$, we have $i_W(q') = 0 < j$, and for all $i$-dimensional forms $\sigma_i$ over $k$, $i_W(q \perp \sigma_i) = i_W\left(\langle a \rangle \perp q' \perp \sigma_i\right) \geq j + 1$. By Lemma \ref{Witt index of sum}, this implies that $i_W(q' \perp \sigma_i) \geq j$ for all $i$-dimensional forms $\sigma_i$ over $k$, hence $\dim q' \geq m_{i,j}(k)$. Thus $m_{i,j+1}(k) = \dim q = \dim q' + 1 \geq m_{i,j}(k) + 1$.
\end{proof}

If $j$ is sufficiently large compared to $i$, then we can find exact expressions for $m_{i,j}(k)$ as we increase~$j$ (see Corollary \ref{exact values for m_{i,j} for large j}).
\begin{lemma}
\label{m_{i,j} for large j}
For any field $k$ of characteristic $\ne 2$ and positive integers $i,j$ such that $j \geq \lceil \frac{i}{2} \rceil \geq 1$, if $m_{i,j+1}(k) < \infty$, then $m_{i,j+1}(k) \geq m_{i,j}(k) + 2$.
\end{lemma}
\begin{proof}
Let $q$ be a quadratic form over $k$ of dimension $m_{i,j+1}(k)$ such that $i_W(q) < j + 1$ and $i_W(q \perp \sigma_i) \geq j+1$ for all $i$-dimensional forms $\sigma_i$ over $k$. Consider the $i$-dimensional form $\widetilde{\sigma}_i$ defined over $k$ by
\[
	\widetilde{\sigma}_i = \begin{cases}
		\frac{i}{2}\HH &\text{ if $i$ is even}, \\
		\langle 1 \rangle \perp \frac{i-1}{2}\HH &\text{ if $i$ is odd}.
	\end{cases}
\]
Then since $i_W\left(q \perp \widetilde{\sigma}_i\right) \geq j + 1$, we conclude $i_W(q) \geq j + 1 - \left\lceil \frac{i}{2}\right\rceil \geq 1$. So $q$ is isotropic, thus $q \simeq \HH \perp q'$ for some form $q'$ over $k$. Moreover, $i_W(q) = 1 + i_W(q') < j + 1$, so $i_W(q') < j$, and for all $i$-dimensional quadratic forms $\sigma_i$ over $k$ we have $i_W(q' \perp \sigma_i) \geq j$ since $i_W(q \perp \sigma_i) \geq j + 1$. Therefore $\dim q' \geq m_{i,j}(k)$ by definition, hence $m_{i,j+1}(k) = \dim q = \dim q' + 2 \geq m_{i,j}(k) + 2$.
\end{proof}
\begin{cor}
\label{exact values for m_{i,j} for large j}
Let $k$ be any field of characteristic $\ne 2$, and let $i$ and $j$ be positive integers such that $j \geq \lceil \frac{i}{2} \rceil \geq 1$. For any integer $s \geq 1$, if $m_{i,j}(k) < \infty$, then $m_{i,j+s}(k) = m_{i,j}(k) + 2s$.
\end{cor}
\begin{proof}
We first observe that $m_{i, j+s}(k) \leq m_{i,j}(k) + 2s$ by Lemma \ref{increasing j in m_{i,j}}. The reverse inequality follows immediately by induction on $s \geq 1$ and Lemma \ref{m_{i,j} for large j}. Hence $m_{i,j+s}(k) = m_{i,j}(k) + 2s$.
\end{proof}
For the majority of this subsection we have studied fields with finite invariants $m_{i,j}$. We conclude this subsection by studying fields with infinite invariants $m_{i,j}$. Recall that $m_{i,j}(k) = \infty$ for a field~$k$ and integers $i,j \geq 1$ if and only if there are no $(i,j)$-realizing forms over $k$.

\begin{lemma}
\label{infinite m_{i,j} for j}
Let $k$ be a field of characteristic $\ne 2$ and suppose there are integers $i^*, j^* \geq 1$ such that $m_{i^*,j^*}(k) = \infty$. Then $m_{i^*, j}(k) = \infty$ for all integers $j \geq 1$.
\end{lemma}
\begin{proof}
By contradiction, suppose there is some integer $j \geq 1$ such that $m_{i^*, j}(k) < \infty$. By our assumption on $j^*$, this implies that $j \ne j^*$, so $j < j^*$ or $j > j^*$. If $j > j^*$, then since $m_{i^*, j}(k) < \infty$, we have $m_{i^*, j^*}(k) \leq m_{i^*, j}(k) < \infty$ by Lemma \ref{initial inequalities}(b). If $j < j^*$ and $m_{i^*, j}(k) < \infty$, then Lemma \ref{increasing j in m_{i,j}} implies $m_{i^*, j^*}(k) \leq m_{i^*, j}(k) + 2(j^* - j) < \infty$. In both cases we have contradicted our assumption that $m_{i^*, j^*}(k) = \infty$. So $m_{i^*,j}(k) = \infty$ for all $j \geq 1$, as desired.
\end{proof}
\begin{lemma}
\label{infinite m_{i,j} for i}
Let $k$ be a field of characteristic $\ne 2$ and suppose there are integers $i^*, j^* \geq 1$ such that $m_{i^*,j^*}(k) = \infty$. Then $m_{i,j^*}(k) = \infty$ for all integers $i \geq 1$.
\end{lemma}
\begin{proof}
By contradiction, assume there is some integer $i \geq 1$ such that $m_{i,j^*}(k) < \infty$. By our assumption on $i^*$, this implies $i \ne i^*$, hence $i < i^*$ or $i > i^*$. If $i < i^*$, then Lemma \ref{initial inequalities}(a) implies that $m_{i^*, j^*}(k) \leq m_{i,j^*}(k) < \infty$. This contradicts our assumption that $m_{i^*, j^*}(k) = \infty$, so $i > i^*$.

Because $m_{i^*,j^*}(k) = \infty$, there are no $(i^*,j^*)$-realizing forms over $k$. So for each quadratic form~$q$ over $k$ with $i_W(q) < j^*$ there is some $i^*$-dimensional form $\sigma_{i^*}$ over $k$ such that $i_W(q \perp \sigma_{i^*}) < j^*$. Similarly, because $i_W(q \perp \sigma_{i^*}) < j^*$, there must be an $i^*$-dimensional form $\sigma_{i^*}'$ over $k$ such that $i_W(q \perp \sigma_{i^*} \perp \sigma_{i^*}') < j^*$. Now let $n$ be any integer such that $n \cdot i^* \geq i$. Then by the above reasoning we can find $n$ quadratic forms $\sigma_1, \ldots, \sigma_n$ over $k$, each with dimension $i^*$, such that $i_W(q \perp \sigma_1 \perp \sigma_2 \perp \cdots \perp \sigma_n) < j^*$. In particular, for all $i$-dimensional subforms $\widetilde{\sigma}_i$ of $\sigma_1 \perp \cdots \perp \sigma_n$ we have $i_W(q \perp \widetilde{\sigma}_i) < j^*$. Since $q$ was an arbitrary quadratic form over $k$ with $i_W(q) < j^*$, we conclude that there are no $(i, j^*)$-realizing forms over $k$, hence $m_{i,j^*}(k) = \infty$, and we have reached a contradiction. Therefore $m_{i,j^*}(k) = \infty$ for all~$i \geq 1$.
\end{proof}

\begin{prop}
\label{if m_{i,j} = infty}
Let $k$ be a field of characteristic $\ne 2$. Then there are integers $i^*, j^* \geq 1$ such that $m_{i^*,j^*}(k) = \infty$ if and only if $m_{i,j}(k) = \infty$ for all integers $i, j \geq 1$.
\end{prop}
\begin{proof}
The reverse implication is trivial. Lemmas~\ref{infinite m_{i,j} for j} and~\ref{infinite m_{i,j} for i} prove the forward implication.
\end{proof}
For example, Proposition \ref{if m_{i,j} = infty} implies $m_{i,j}(\mathbb{R}) = \infty$ for all integers $i, j \geq 1$ since $m_{1,1}(\mathbb{R}) = \infty$. We also note that Proposition \ref{if m_{i,j} = infty} gives a positive answer to the question posed in Remark~\ref{do the m_{i,1} determine m_{i,j}} if one of the invariants $m_{i,1}$ is infinite.

\subsection{Refined $m$-invariants of complete discretely valued fields} 
\label{refined m of cdvf}
Let $K$ be a complete discretely valued field with residue field $k$ of characteristic $\ne 2$. Using Springer's Theorem, one can show that $u(K) = 2u(k)$ and $m(K) = 2m(k)$, and the proof of these facts suggests that one should view these as additive identities; i.e., $u(K) = u(k) + u(k)$ and $m(K) = m(k) + m(k)$. We now want to use this viewpoint to compute the refined $m$-invariants $m_{i,1}(K)$ for $i \geq 1$ in terms of the refined $m$-invariants of $k$. Before proceeding, we observe that for any $i \geq u(K)$, Lemma \ref{upper bound on m_{i,j}} implies that $1 \leq m_{i,1}(K) \leq \max\{1, u(K) + 1 - i\} = 1$, so $m_{i,1}(K) = 1$. Therefore we need to focus only on integers $i$ such that $1 \leq i < u(K)$. Throughout this subsection we will assume that $u(k) < \infty$, as this assumption guarantees $m_{r,1}(k) < \infty$ for all $1 \leq r \leq u(k)$ by Lemma \ref{upper bound on m_{i,j}}. In particular, for all $1 \leq r \leq u(k)$ an $(r,1)$-realizing form exists over $k$.

Throughout this subsection we will use the following terminology. Let $K$ be a complete discretely valued field with valuation ring $\mathcal{O}_K$ and residue field $k$ of characteristic $\ne 2$. Let $\overline{a}_1, \ldots, \overline{a}_n$ be any non-zero elements of $k$, and consider the quadratic form $\overline{q} = \langle \overline{a}_1, \ldots, \overline{a}_n \rangle$ over $k$. For each $i = 1, \ldots, n$ let $a_i \in \mathcal{O}_K^{\times}$ be a unit in $\mathcal{O}_K$ whose image in $k$ is $\overline{a}_i$. Then we call the quadratic form $q = \langle a_1, \ldots, a_n \rangle$ over $K$ a \textit{lift} of $\overline{q}$ to $K$.

We begin by proving a lemma that will be used at several points of this subsection.
\begin{lemma}
\label{compensating dimension}
    Let $K$ be a complete discretely valued field with uniformizer $\pi$ and residue field~$k$ of characteristic $\ne 2$ such that $2 \leq u(k) < \infty$. Let $i$ be any integer such that $i \geq 2$, and let $q \simeq q_1 \perp \pi \cdot q_2$ be an $(i,1)$-realizing form over $K$, where the entries of $q_1$ and $q_2$ are all units in the valuation ring of $K$. If $1 \leq \dim q_1 < m_{s-1,1}(k)$ for some integer $s$ such that $2 \leq s \leq \min\{i, u(k)\}$ and $\dim q_2 \geq 1$, then $\dim q_2 \geq m_{i-s+1,1}(k)$.
\end{lemma}
\begin{proof}
    If $i - s + 1 \geq u(k)$, then $m_{i-s+1,1}(k) = 1$, so the claim follows immediately since $\dim q_2 \geq 1$.
    
    Now assume $i - s + 1 < u(k)$. Since $q$ is an $(i,1)$-realizing form over $K$, it follows that $q$ is anisotropic over $K$. So by Springer's Theorem, both residue forms $\overline{q}_1, \overline{q}_2$ are anisotropic over~$k$. Moreover, $\dim \overline{q}_1 < m_{s-1,1}(k)$, so there must be an $(s-1)$-dimensional quadratic form $\overline{\sigma}_{s-1}$ over~$k$ such that $\overline{q}_1 \perp \overline{\sigma}_{s-1}$ is anisotropic over $k$. Now let $\overline{\sigma}_{i-s+1}$ be any $(i-s+1)$-dimensional quadratic form over $k$. Then for any lifts $\sigma_{s-1}, \sigma_{i-s+1}$ of $\overline{\sigma}_{s-1}, \overline{\sigma}_{i-s+1}$, respectively, to $K$, the quadratic form $q \perp (\sigma_{s-1} \perp \pi \cdot \sigma_{i-s+1})$ must be isotropic over $K$ since $q$ is an $(i,1)$-realizing form. Because $\overline{q}_1 \perp \overline{\sigma}_{s-1}$ is anisotropic over $k$, this implies that $\overline{q}_2 \perp \overline{\sigma}_{i-s+1}$ must be isotropic over $k$ by Springer's Theorem. The form $\overline{\sigma}_{i-s+1}$ was arbitrary, so we conclude that~$\overline{q}_2$ is an $(i-s+1,1)$-realizing form over $k$. Thus $\dim q_2 = \dim \overline{q}_2 \geq m_{i-s+1,1}(k)$.
\end{proof}

In the above notation, we will first compute $m_{i,1}(K)$ for integers $i$ such that $1 \leq i \leq u(k)$ (see Proposition \ref{m_{i,1} of cdvf for small i}).
\begin{lemma}
    \label{upper bound for m_{i,1} for small i}
Let $K$ be a complete discretely valued field with residue field $k$ of characteristic $\ne 2$ such that $u(k) < \infty$. Then for any integers $i$ and $r$ such that $1 \leq r \leq i \leq u(k)$, we have $m_{i,1}(K) \leq m_{r,1}(k) + m_{i-r+1,1}(k)$.
\end{lemma}
\begin{proof}
    Let $i,r$ be given. Let $\overline{q}_1$ be an $m_{r,1}(k)$-dimensional $(r,1)$-realizing form over $k$ and let $\overline{q}_2$ be an $m_{i-r+1,1}(k)$-dimensional $(i-r+1,1)$-realizing form over $k$. Let $q_1, q_2$ be lifts of $\overline{q}_1, \overline{q}_2$, respectively, to $K$, and for a uniformizer $\pi$ of $K$, consider the quadratic form $q$ over $K$ defined by $q = q_1 \perp \pi \cdot q_2$. Springer's Theorem implies that $q$ is anisotropic over $K$ since $\overline{q}_1$ and $\overline{q}_2$ are anisotropic over $k$. 
    
    Now let $\sigma_i \simeq \sigma_{i,1} \perp \pi \cdot \sigma_{i,2}$ be any $i$-dimensional quadratic form over $K$, where the entries of~$\sigma_{i,1}$ and $\sigma_{i,2}$ are all units in the valuation ring of $K$. If $r \leq \dim \sigma_{i,1}$, then by our choice of~$\overline{q}_1$, for any $r$-dimensional subform $\sigma_{r,1}$ of $\sigma_{i,1}$ we have $i_W(\overline{q}_1 \perp \overline{\sigma}_{r,1}) \geq 1$. Thus $i_W(\overline{q}_1 \perp \overline{\sigma}_{i,1}) \geq 1$, hence $i_W(q \perp \sigma_i) \geq 1$ by Springer's Theorem. If $r > \dim \sigma_{i,1}$, then $\dim \sigma_{i,2} =$~$i - \dim \sigma_{i,1} > i - r$, therefore $\dim \sigma_{i,2} \geq i - r + 1$. So by our choice of $\overline{q}_2$, $i_W(\overline{q}_2 \perp \overline{\sigma}_{i,2}) \geq 1$, thus $i_W(q \perp \sigma_i) \geq 1$. The form $\sigma_i$ was arbitrary, so $q$ is an $(i,1)$-realizing form over $K$, hence $m_{i,1}(K) \leq \dim q = m_{r,1}(k) + m_{i-r+1,1}(k)$.
\end{proof}

\begin{lemma}
    \label{preliminary lower bound}
Let $K$ be a complete discretely valued field with uniformizer $\pi$ and residue field~$k$ of characteristic $\ne 2$ such that $u(k) < \infty$. Let $i$ be any integer such that $1 \leq i \leq u(k)$, and let $q \simeq q_1 \perp \pi \cdot q_2$ be any $(i,1)$-realizing form over $K$, where the entries of $q_1$ and $q_2$ are all units in the valuation ring of $K$. Then $\dim q_1, \dim q_2 \geq m_{i,1}(k)$.
\end{lemma}
\begin{proof}
    We first observe that since the form $q$ is anisotropic over $K$, both residue forms $\overline{q}_1, \overline{q}_2$ are anisotropic over $k$ by Springer's Theorem. Next, by contradiction, assume the lemma is false, and without loss of generality, assume $\dim q_1 < m_{i,1}(k)$. 

    The form $\overline{q}_1$ is anisotropic over $k$ with $\dim \overline{q}_1 < m_{i,1}(k)$, so there must be an $i$-dimensional quadratic form $\overline{\sigma}_i$ over $k$ such that $\overline{q}_1 \perp \overline{\sigma}_i$ is anisotropic over $k$. Let $\sigma_i$ be any lift of $\overline{\sigma}_i$ to $K$. Then since $\overline{q}_2$ is anisotropic over $k$, it follows that $q \perp \sigma_i \simeq (q_1 \perp \sigma_i) \perp \pi \cdot q_2$ is anisotropic over~$K$. This is a contradiction of $q$ being an $(i,1)$-realizing form over $K$, so the proof is complete.
\end{proof}

\begin{lemma}
    \label{lower bound for m_{i,1} for small i}
Let $K$ be a complete discretely valued field with residue field $k$ of characteristic $\ne 2$ such that $u(k) < \infty$. If $i$ is any integer such that $1 \leq i \leq u(k)$, then 
\[
    m_{i,1}(K) \geq \min_{1 \leq r \leq i} \{m_{r,1}(k) + m_{i-r+1,1}(k)\}.
\]
\end{lemma}
\begin{proof}
    If $i = 1$, then $m_{i,1}(K) = m(K) = m(k) + m(k) = \min_{1 \leq r \leq i}\{m_{r,1}(k) + m_{i-r+1,1}(k)\}$. 
    
    Now assume $i \geq 2$ (which implies $u(k) \geq 2$). We will prove the claim by showing that any $(i,1)$-realizing form $q$ over $K$ must have $\dim q \geq \min_{1 \leq r \leq i} \{m_{r,1}(k) + m_{i-r+1,1}(k)\}$. To that end, let~$\pi$ be a uniformizer for $K$ and let $q \simeq q_1 \perp \pi \cdot q_2$ be any $(i,1)$-realizing form over $K$, where the entries of~$q_1$ and $q_2$ are all units in the valuation ring of $K$. Then by Lemma~\ref{preliminary lower bound}, $\dim q_1, \dim q_2 \geq m_{i,1}(k)$. Furthermore, because $q$ is an $(i,1)$-realizing form, the form $\pi \cdot q$ is also an $(i,1)$-realizing form. So, without loss of generality, we may assume $\dim q_1 \leq \dim q_2$. We now consider two cases for~$\dim q_1$.

    First, assume that $\dim q_1 \geq m_{1,1}(k)$. Then $\dim q_2 \geq m_{1,1}(k)$, and since $m_{t,1}(k) \geq m_{t',1}(k)$ for any integers $t' \geq t \geq 1$ (Lemma \ref{initial inequalities}(a)), we have $\dim q \geq 2m_{1,1}(k) \geq \min_{1 \leq r \leq i}\{m_{r,1}(k) +$~$ m_{i-r+1,1}(k)\}$.

    Next, assume $\dim q_1 < m_{1,1}(k)$. Then because $\dim q_1 \geq m_{i,1}(k)$, there must be some integer~$s$ such that $2 \leq s \leq i$ and $m_{s,1}(k) \leq \dim q_1 < m_{s-1,1}(k)$. By Lemma \ref{compensating dimension} we conclude that $\dim q_2 \geq m_{i-s+1,1}(k)$. Thus $\dim q \geq m_{s,1}(k) + m_{i-s+1,1}(k) \geq \min_{1 \leq r \leq i}\{m_{r,1}(k) + m_{i-r+1,1}(k)\}$.

    In both of the above cases, $\dim q \geq \min_{1 \leq r \leq i} \{m_{r,1}(k) + m_{i-r+1,1}(k)\}$, so the proof is complete.
\end{proof}
\begin{prop}
    \label{m_{i,1} of cdvf for small i}
Let $K$ be a complete discretely valued field with residue field $k$ of characteristic $\ne 2$ such that $u(k) < \infty$. If $i$ is any integer such that $1 \leq i \leq u(k)$, then
\[
    m_{i,1}(K) = \min_{1 \leq r \leq i}\{m_{r,1}(k) + m_{i-r+1,1}(k)\}.
\]
\end{prop}
\begin{proof}
    By Lemma \ref{upper bound for m_{i,1} for small i}, for any integer $r$ such that $1 \leq r \leq i$, $m_{i,1}(K) \leq m_{r,1}(k) + m_{i-r+1,1}(k)$. Therefore $m_{i,1}(K) \leq \min_{1 \leq r \leq i}\{m_{r,1}(k) + m_{i-r+1,1}(k)\}$. The opposite inequality was shown in Lemma \ref{lower bound for m_{i,1} for small i}, so we have the desired equality.
\end{proof}
\begin{remark}
\label{what if m=u, small i}
    Let $K$ be a complete discretely valued field with residue field $k$ of characteristic~$\ne 2$ with $m(k) = u(k) < \infty$. Then $m(K) = u(K) < \infty$ and for any $1 \leq i \leq u(k)$, $m_{i,1}(K) =$~$m(K) + 1 - i$ by Theorem \ref{upper and lower bounds}. Furthermore, Theorem \ref{upper and lower bounds} also implies that for any $1 \leq r \leq i \leq u(k)$, $m_{r,1}(k) = m(k) + 1 - r$ and $m_{i-r+1,1}(k) = m(k) + r - i$. Thus for any $1 \leq r \leq i \leq u(k)$,
    \[
        m_{r,1}(k) + m_{i-r+1,1}(k) = (m(k) + 1 - r) + (m(k) + r - i) = 2m(k) + 1 - i = m(K) + 1 - i.
    \]
    So $\min_{1 \leq r \leq i}\{m_{r,1}(k) + m_{i-r+1,1}(k)\} = m(K) + 1 - i$.
\end{remark}
In contrast with Remark \ref{what if m=u, small i}, the next example shows that in the context of Proposition \ref{m_{i,1} of cdvf for small i}, if $m(k) < u(k)$, then the quantities used to compute $m_{i,1}(K)$ may not all be equal.
\begin{example}
    Let $K$ be a complete discretely valued field with residue field $k$ of characteristic $\ne 2$ such that $m(k) = 6$, $u(k) = 8$, and $m_{5,1}(k) = 2$ (this is the smallest possible value of $m_{5,1}(k)$ by Corollary \ref{decreasing i to 1}). For such a field $k$ we have $m_{8,1}(k) = 1$, $m_{i,1}(k) = 2$ for $i = 6, 7$, and $m_{i,1}(k) = 7-i$ for $1 \leq i \leq 5$. Hence $m_{1,1}(k) + m_{8,1}(k) = m_{2,1}(k) + m_{7,1}(k) =$~$7$, $m_{3,1}(k) + m_{6,1}(k) = 6$, and $m_{4,1}(k) + m_{5,1}(k) = 5$. Therefore Proposition \ref{m_{i,1} of cdvf for small i} implies that $m_{8,1}(K) = 5$.
\end{example}

We will now focus on computing $m_{i,1}(K)$ for integers $i$ such that $u(k) < i < u(K)$. This computation will be carried out in the next two lemmas.
\begin{lemma}
    \label{upper bound for m_{i,1} for large i}
Let $K$ be a complete discretely valued field with residue field $k$ of characteristic $\ne 2$ such that $u(k) < \infty$. If $i$ is any integer such that $u(k) < i < u(K)$, then
\[
    m_{i,1}(K) \leq \min_{\lceil \frac{i}{2} \rceil \leq s \leq u(k)}\{m_{s, 1}(k) + m_{i-s+1,1}(k), m_{i-u(k), 1}(k)\}.
\]
\end{lemma}
\begin{proof}
    Let $m_* = \min_{\lceil \frac{i}{2} \rceil \leq s \leq u(k)}\{m_{s, 1}(k) + m_{i-s+1,1}(k), m_{i-u(k), 1}(k)\}$. To prove the lemma, we will construct an $m_*$-dimensional $(i,1)$-realizing form over $K$.

    First, suppose that $m_* = m_{i-u(k), 1}(k)$. Let $\overline{q}_1$ be an $m_*$-dimensional $(i-u(k), 1)$-realizing form over $k$, and let $q_1$ be a lift of $\overline{q}_1$ to $K$. Let $\sigma_i \simeq \sigma_{i,1} \perp \pi \cdot \sigma_{i,2}$ be an arbitrary $i$-dimensional quadratic form over $K$, where~$\pi$ is a uniformizer for $K$ and the entries of $\sigma_{i,1}$ and $\sigma_{i,2}$ are all units in the valuation ring of $K$, and consider the form $q_1 \perp \sigma_i$ over $K$. If $\dim \sigma_{i,2} > u(k)$, then $\overline{\sigma}_{i,2}$ is isotropic over~$k$ by the definition of $u(k)$. This then implies $\sigma_{i,2}$ is isotropic over $K$, hence $q_1 \perp \sigma_i$ is isotropic over $K$ as well. If $\dim \sigma_{i,2} \leq u(k)$, then $\dim \sigma_{i,1} = i - \dim \sigma_{i,2} \geq i-u(k)$, so $\overline{q}_1 \perp \overline{\sigma}_{i,1}$ is isotropic over~$k$ by our choice of $\overline{q}_1$. Thus $q_1 \perp \sigma_i$ is isotropic over $K$. Since $\overline{q}_1$ is anisotropic over~$k$, $q_1$ is anisotropic over $K$, so $q_1$ is an $m_*$-dimensional $(i,1)$-realizing form over $K$.

    Now suppose $m_* = m_{s_*,1}(k) + m_{i-s_*+1,1}(k)$ for some $\lceil \frac{i}{2} \rceil \leq s_* \leq u(k)$. Let $\overline{q}_1$ be an $m_{s_*,1}(k)$-dimensional $(s_*,1)$-realizing form over $k$, and let $\overline{q}_2$ be an $m_{i-s_*+1,1}(k)$-dimensional $(i - s_* + 1, 1)$-realizing form over $k$. Let $q_1, q_2$ be lifts of $\overline{q}_1, \overline{q}_2$, respectively, to $K$, and let $q = q_1 \perp \pi \cdot q_2$. Both~$\overline{q}_1$ and $\overline{q}_2$ are anisotropic over $k$, so $q$ is anisotropic over $K$. Now let $\sigma_i \simeq \sigma_{i,1} \perp \pi \cdot \sigma_{i,2}$ be any $i$-dimensional quadratic form over $K$, where the entries of $\sigma_{i,1}$ and $\sigma_{i,2}$ are all units in the valuation ring of $K$. If $\dim \sigma_{i,1} \geq s_*$, then by our choice of $\overline{q}_1$, the form $\overline{q}_1 \perp \overline{\sigma}_{i,1}$ is isotropic over $k$, hence $q \perp \sigma_i$ is isotropic over $K$. If $\dim \sigma_{i,1} < s_*$, then $\dim \sigma_{i,2} = i - \dim \sigma_{i,1} \geq i - s_* + 1$. So by our choice of~$\overline{q}_2$, the form $\overline{q}_2 \perp \overline{\sigma}_{i,2}$ is isotropic over $k$, thus $q \perp \sigma_i$ is isotropic over $K$. Since $\sigma_i$ was arbitrary, we conclude that the $m_*$-dimensional form $q = q_1 \perp \pi \cdot q_2$ is an $(i,1)$-realizing form over~$K$, and the proof is complete.
\end{proof}
\begin{lemma}
    \label{lower bound for m_{i,1} for large i}
Let $K$ be a complete discretely valued field with residue field $k$ of characteristic $\ne 2$ such that $u(k) < \infty$. If $i$ is any integer such that $u(k) < i < u(K)$, then
\[
    m_{i,1}(K) \geq \min_{\lceil \frac{i}{2} \rceil \leq s \leq u(k)}\{m_{s,1}(k) + m_{i-s+1,1}(k), m_{i-u(k), 1}(k)\}.
\]
\end{lemma}
\begin{proof}
    To prove the lemma, we will show that if $q$ is any $(i,1)$-realizing form over $K$, then either $\dim q \geq m_{i-u(k),1}(k)$ or $\dim q \geq m_{s_*,1}(k) + m_{i-s_*+1,1}(k)$ for some $\lceil \frac{i}{2} \rceil \leq s_* \leq u(k)$. 

    To that end, let $q \simeq q_1 \perp \pi \cdot q_2$ be any $(i,1)$-realizing form over $K$, where $\pi$ is a uniformizer for~$K$ and the entries of $q_1$ and $q_2$ are all units in the valuation ring of $K$. We note here that since~$q$ is anisotropic over $K$, both residue forms $\overline{q}_1, \overline{q}_2$ are anisotropic over $k$ by Springer's Theorem.

    We first claim that $\dim q_1 \geq m_{\lceil \frac{i}{2} \rceil, 1}(k)$ or $\dim q_2 \geq m_{\lceil \frac{i}{2} \rceil, 1}(k)$. Indeed, by contradiction, assume $\dim q_1, \dim q_2 < m_{\lceil \frac{i}{2} \rceil, 1}(k)$. Then since $\overline{q}_1$ and $\overline{q}_2$ are anisotropic over $k$, this assumption on the forms' dimensions implies that there must be quadratic forms $\overline{\sigma}_1, \overline{\sigma}_2$ over $k$, both with dimension~$\lceil \frac{i}{2} \rceil$, such that $\overline{q}_1 \perp \overline{\sigma}_1$ and $\overline{q}_2 \perp \overline{\sigma}_2$ are anisotropic over $k$. Then for lifts $\sigma_1, \sigma_2$ of $\overline{\sigma}_1, \overline{\sigma}_2$, respectively, to~$K$, the form $q \perp (\sigma_1 \perp \pi \cdot \sigma_2) \simeq (q_1 \perp \sigma_1) \perp \pi \cdot (q_2 \perp \sigma_2)$ is anisotropic over $K$. However, the form $\sigma_1 \perp \pi \cdot \sigma_2$ has dimension $2 \cdot \lceil \frac{i}{2} \rceil \geq i$, so this is a contradiction of $q$ being an $(i,1)$-realizing form over $K$. This proves the claim.
    
    So, without loss of generality, we may assume that $\dim q_2 \geq m_{\lceil \frac{i}{2} \rceil, 1}(k)$. We now consider three cases for $\dim q_1$.

    First, suppose $\dim q_1 \geq m_{\lceil \frac{i}{2} \rceil, 1}(k)$. By Lemma \ref{initial inequalities}(a), we have $m_{\lceil \frac{i}{2} \rceil, 1}(k) \geq m_{i-\lceil \frac{i}{2} \rceil + 1, 1}(k)$ since $\lceil \frac{i}{2} \rceil \leq i - \lceil \frac{i}{2} \rceil + 1$. Therefore, since $\dim q = \dim q_1 + \dim q_2$ and $\dim q_1, \dim q_2 \geq m_{\lceil \frac{i}{2} \rceil, 1}(k)$, we have $\dim q \geq m_{\lceil \frac{i}{2} \rceil, 1}(k) + m_{i - \lceil \frac{i}{2} \rceil + 1, 1}(k)$, completing the proof in this case.

    Next, suppose $\dim q_1 = 0$. If $\overline{\sigma}_{u(k)}$ is any anisotropic $u(k)$-dimensional form over $k$ and~$\sigma_{u(k)}$ is any lift of $\overline{\sigma}_{u(k)}$ to $K$, then the form $q \perp \sigma_{u(k)} \simeq \sigma_{u(k)} \perp \pi \cdot q_2$ is anisotropic over $K$. Because~$q$ is an $(i,1)$-realizing form over $K$, this then implies that for any $(i-u(k))$-dimensional form $\overline{\sigma}_{i-u(k)}$ over~$k$, the quadratic form $\overline{q}_2 \perp \overline{\sigma}_{i-u(k)}$ must be isotropic over $k$. Thus $\overline{q}_2$ is an $(i-u(k), 1)$-realizing form over $k$, hence $\dim q = \dim q_2 = \dim \overline{q}_2 \geq m_{i-u(k), 1}(k)$. This completes the proof in this case.

    Finally, suppose $1 \leq \dim q_1 < m_{\lceil \frac{i}{2} \rceil, 1}(k)$. We note here that since $m_{u(k), 1}(k) = 1$, this assumption on $\dim q_1$ implies that $\lceil \frac{i}{2} \rceil < u(k)$, hence $u(k) \geq 2$. We also conclude that there is some integer~$s_*$ such that $\lceil \frac{i}{2} \rceil < s_* \leq u(k)$ and $m_{s_*,1}(k) \leq \dim q_1 < m_{s_*-1,1}(k)$. Since $i > u(k) \geq 2$, it follows that $i \geq 2$ and $2 \leq s_* \leq u(k)$, so Lemma \ref{compensating dimension} implies that $\dim q_2 \geq m_{i-s_*+1,1}(k)$. Therefore $\dim q = \dim q_1 + \dim q_2 \geq m_{s_*,1}(k) + m_{i-s_*+1,1}(k)$, and the proof is complete.
\end{proof}

Combining Lemmas \ref{upper bound for m_{i,1} for large i} and \ref{lower bound for m_{i,1} for large i} proves the following result.
\begin{prop}
    \label{m_{i,1} of cdvf for large i}
Let $K$ be a complete discretely valued field with residue field $k$ of characteristic $\ne 2$ such that $u(k) < \infty$. If $i$ is any integer such that $u(k) < i < u(K)$, then
\[
    m_{i,1}(K) = \min_{\lceil \frac{i}{2} \rceil \leq s \leq u(k)} \{m_{s,1}(k) + m_{i-s+1,1}(k), m_{i-u(k),1}(k)\}.
\]
\end{prop}
\begin{remark}
    Let $K$ be a complete discretely valued field with residue field $k$ of characteristic $\ne 2$ such that $m(k) = u(k) < \infty$, and let $i$ be any integer such that $u(k) < i < u(K)$. Then $m(K) = u(K) <$~$\infty$ and $m_{i,1}(K) = m(K) + 1 - i$. Furthermore, for any $\lceil \frac{i}{2} \rceil \leq s \leq u(k)$ we have
    \[
        m_{s,1}(k) + m_{i-s+1,1}(k) = (m(k) + 1 - s) + (m(k) + s - i) = m(k) + m(k) + 1 - i = m(K) + 1 - i.
    \]
    We also have $m_{i-u(k), 1}(k) = m(k) + 1 - (i - u(k)) = m(k) + m(k) + 1 - i = m(K) + 1 - i$.
    So $\min_{\lceil \frac{i}{2} \rceil \leq s \leq u(k)}\{m_{s,1}(k) + m_{i-s+1,1}(k), m_{i-u(k),1}(k)\} = m(K) + 1 - i$.
\end{remark}
We conclude this subsection with the following example.
\begin{example}
    Let $K$ be a complete discretely valued field with residue field $k$ of characteristic $\ne 2$ such that $u(k) < \infty$ and $m(k) = 2$. Then $m_{u(k), 1}(k) = 1$ and Lemmas \ref{initial inequalities}(a) and \ref{lower bound on m_{i,1} for small i} imply that $m_{r, 1}(k) = 2$ for all $1 \leq r < u(k)$. In particular, if $u(k) < i < u(K)$, then $m_{i - u(k), 1}(k) = 2$. So by Lemma \ref{upper bound on m_{i,j}} and Propositions \ref{m_{i,1} of cdvf for small i} and \ref{m_{i,1} of cdvf for large i}, we conclude
    \[
        m_{i,1}(K) = \begin{cases}
            4 &\text{ if } 1 \leq i < u(k), \\
            3 &\text{ if } i = u(k), \\
            2 &\text{ if } u(k) < i < u(K), \\
            1 &\text{ if } i \geq u(K).
        \end{cases}
    \]
    From this we see that if $u(k) \geq 4$, then there are integers $i < u(K)$ such that $m_{i,1}(K) \ne m(K) + 1 - i$ (e.g., $m_{3,1}(K) = 4 \ne 2 = m(K) +1 - 3$).
\end{example}

\subsection{The refined $m$-invariants separate fields}
\label{refined m separates} 
To conclude this section, we investigate whether these refined $m$-invariants can distinguish fields that are not distinguished by the $u$- and $m$-invariant. In other words,
\begin{question}
\label{separate fields?}
    Are there fields $k_1, k_2$ with $u(k_1) = u(k_2)$, $m(k_1) = m(k_2)$, but $m_{i,j}(k_1) \ne$~$m_{i,j}(k_2)$ for some integers $i,j$?
\end{question}
We will now show that one can construct fields that give an affirmative answer to Question~\ref{separate fields?} (see Corollary \ref{m_{i,j} separates}), and note that Theorem \ref{upper and lower bounds} implies that any fields $k_1, k_2$ providing a positive answer to this question must satisfy $m(k_1) = m(k_2) < u(k_1) = u(k_2)$. We begin by proving a lemma about non-real linked fields with $m$-invariant 6 (see, e.g., \cite[p.~370]{lam} for the definition of a linked field).

\begin{lemma}
\label{linked field with m=6}
Let $k$ be a non-real linked field of characteristic $\ne 2$ such that $m(k) = 6$. Then $u(k) = 8$ and 
\[
	m_{i,1}(k) = \begin{cases}
	6 &\text{ if } 1 \leq i \leq 3, \\
	9 - i &\text{ if } 4 \leq i \leq 7, \\
	1 &\text{ if } i \geq 8.
	\end{cases}
\]
\end{lemma}
\begin{proof}
Since $k$ is a non-real linked field, we have $u(k) = 1, 2, 4,$ or $8$ \cite[Theorem XI.6.21]{lam}. Because $m(k) = 6$, \cite[Corollary~1.6]{m-inv} implies that $u(k) \geq 8$. Therefore $u(k) = 8$, and by Lemma \ref{upper bound on m_{i,j}} we conclude that $m_{i,1}(k) = 1$ for all $i \geq u(k) = 8$.

By assumption, $m_{1,1}(k) = m(k) = 6$, so we next consider $i = 2, 3$. We first show that there are no five-dimensional $(2,1)$- or $(3,1)$-realizing forms over $k$. Indeed, let $q$ be an arbitrary anisotropic quadratic form over $k$ with $\dim q = 5$. The field $k$ is linked, so $q$ is a Pfister neighbor \cite[Theorem~X.4.20]{lam}. Therefore there is some $a \in k^{\times}$, some 3-fold Pfister form $\varphi$ over $k$, and some three-dimensional form $\sigma_3$ over $k$ such that $a \cdot q \perp \sigma_3 \simeq \varphi$. Furthermore, $q$ is isotropic if and only if~$\varphi$ is isotropic \cite[Proof of Proposition~X.4.17]{lam}, so $\varphi$ is anisotropic over $k$ since $q$ is anisotropic over $k$. Therefore $a \cdot \varphi \simeq q \perp a \cdot \sigma_3$ is anisotropic over $k$. This implies that $q$ is not a $(2,1)$-realizing form, nor is it a $(3,1)$-realizing form over $k$.

Now, since $m_{1,1}(k) = 6$, by Lemma \ref{initial inequalities}(a) and Corollary \ref{decreasing i to 1} we conclude $5 \leq m_{2,1}(k) \leq 6$. We just showed that there are no five-dimensional $(2,1)$-realizing forms over $k$, thus $m_{2,1}(k) = 6$. This then implies that $5 \leq m_{3,1}(k) \leq 6$, and since there are no five-dimensional $(3,1)$-realizing forms over $k$, we conclude that $m_{3,1}(k) = 6$. 

Finally, consider $4 \leq i \leq 7$. Since $m_{8,1}(k) = 1$, we conclude $m_{i,1}(k) \leq m_{8,1}(k) + (8-i) = 9-i$ by Proposition \ref{decreasing i in m_{i,j}}. Next, since $m_{3,1}(k) = 6$, Proposition \ref{decreasing i in m_{i,j}} implies $m_{i,1}(k) \geq m_{3,1}(k) - (i - 3) = 9-i$. Therefore $m_{i,1}(k) = 9-i$ for $4 \leq i \leq 7$, as desired.
\end{proof}

We now construct a field $F$ of characteristic $\ne 2$ with $u(F) = 8, m(F) = 6$, and $m_{2,1}(F) = 5$ using a process very similar to the construction found in the proof of \cite[Proposition~4.3]{6-dim}. First, we recall some terminology. Let $k$ be any field of characteristic $\ne 2$, and let $q$ be a regular quadratic form over $k$ with $\dim q \geq 2$ such that $q \not\simeq \mathbb{H}$. Then the \textit{function field} of $q$, denoted by~$k(q)$, is the function field of the projective variety defined over $k$ by $q = 0$. The form $q$ is isotropic over~$k(q)$, $k(q)$ has transcendence degree $\dim q - 2$ over~$k$, and $k(q)$ is purely transcendental over $k$ if and only if $q$ is isotropic over $k$ (see, e.g., \cite[p.~463]{function fields}). For a family $\{q_{\alpha}\}$ of quadratic forms~$q_{\alpha}$ over~$k$ with $\dim q_{\alpha} \geq 3$ for all $\alpha$, the field $k(\{q_{\alpha}\})$ is the free compositum of the fields $k(q_{\alpha})$ for all~$\alpha$ \cite[p.~333]{lam}.

\begin{prop}
\label{merkurjev tower}
There is a field $F$ of characteristic $\ne 2$ such that $u(F) = 8$, $m(F) = 6$, and $m_{2,1}(F) = 5$.
\end{prop}
\begin{proof}
Let $E$ be any field of characteristic $\ne 2$ over which there is an anisotropic six-dimensional quadratic form $\psi$ with $\det \psi = -1 \in E^{\times} / E^{\times 2}$ (e.g., $E = \mathbb{C}(x,y,z)$). After scaling $\psi$ by its first entry if necessary, without loss of generality, we may assume $\psi$ represents $1$. So $\psi \simeq \langle 1 \rangle \perp \psi'$ for some five-dimensional form~$\psi'$ over $E$ with $\det \psi' = -1$. We now make two observations. The first is that for any field extension $K / E$, $\det \psi_K = \det \psi'_K = -1 \in K^{\times} / K^{\times 2}$. The second is that for any field extension $K/E$, $\psi_K$ is isotropic over~$K$ if and only if $\psi'_K$ is a Pfister neighbor over $K$. Indeed, $\psi_K \simeq \langle 1 \rangle \perp \psi'_K$ is isotropic over $K$ if and only if $\psi'_K$ represents $-1$ \cite[Corollary~I.3.5]{lam}, and since $\dim \psi'_K = 5$ and $\det \psi'_K = -1$, the form $\psi'_K$ represents $-1$ if and only if $\psi'_K$ is a Pfister neighbor \cite[Proposition~X.4.19]{lam}. 

We now define fields $E_n$ for $n \geq 0$ inductively as follows. Let $\widetilde{E}_n = E_n(X_n)$ for an indeterminate~$X_n$. Let
\begin{align*}
	E_0 &= E \\
	E_{2i + 1} &= \widetilde{E}_{2i}\left(\{\varphi \mid \varphi \text{ a regular 9-dimensional quadratic form over $E_{2i}$}\}\right) \text{ for $i \geq 0$} \\
	E_{2i} &= \widetilde{E}_{2i-1}\left(\{\psi' \perp \langle c, d \rangle \mid c, d \in E_{2i-1}^{\times}\}\right) \text{ for $i \geq 1$}.
\end{align*}
Let $F = \bigcup_{n = 0}^{\infty} E_n$. For any $9$-dimensional quadratic form $\varphi$ over $F$, $\varphi$ is already defined over $E_{2i}$ for some $i \geq 0$. Therefore $\varphi$ becomes isotropic over $E_{2i+1}$, hence over $F$. Thus $u(F) \leq 8$. 

Next, recall that any anisotropic quadratic form remains anisotropic over purely transcendental extensions \cite[Lemma~IX.1.1]{lam}. Furthermore, any anisotropic quadratic form of dimension $\leq 8$ remains anisotropic over the function field of any quadratic form of dimension $\geq 9$ \cite[Theorem~1]{function fields}. We also know that any anisotropic six-dimensional quadratic form with determinant~$-1$ remains anisotropic over the function field of any quadratic form of dimension $\geq 7$ \cite[Lemma~1.1, Proposition~5.6]{mer}, \cite[Theorem~XIII.2.6]{lam}. Since $\psi$ is anisotropic over $E$, these facts then imply that $\psi$ remains anisotropic over $F$. Moreover, since $\psi$ is anisotropic over $F$, $\psi$ is anisotropic over $K$ for any intermediate field $E \subseteq K \subseteq F$. This implies that, for any intermediate field $E \subseteq K \subseteq F$, $\psi'$ is anisotropic over $K$ and, by the first paragraph of the proof, $\psi'$ is not a Pfister neighbor over $K$. 

Now let $\sigma_2 \simeq \langle c, d \rangle$ be an arbitrary two-dimensional quadratic form over $F$. Then $\sigma_2$ is already defined over $E_{2i-1}$ for some $i \geq 1$, so $\psi' \perp \sigma_2$ is isotropic over $E_{2i}$, hence isotropic over $F$. In particular, for all $a \in F^{\times}$ the form $\psi' \perp \langle 1, -a \rangle \simeq \psi \perp \langle -a \rangle$ is isotropic over $F$. Since $\psi'$ and $\psi$ are anisotropic over $F$, we conclude that $\psi$ is a $(1,1)$-realizing form over $F$, and $\psi'$ is a $(2,1)$-realizing form over $F$. Thus $m(F) \leq 6$ and $m_{2,1}(F) \leq 5$. 

As we next explain, to complete the proof it suffices to show that $m(F) > 4$. Indeed, if $m(F) > 4$, then $5 \leq m(F) \leq 6$, and since $m(F)$ cannot equal 5 \cite[1.1a), p.~194]{m-inv}, we conclude $m(F) = 6$. Next, since $m_{2,1}(F) \leq 5$ and $m(F) = 6$, Corollary \ref{decreasing i to 1} implies $m_{2,1}(F) \geq 5$, so $m_{2,1}(F) = 5$. Finally, $u(F) \leq 8$ by construction, and since $m(F) = 6$, we must have $u(F) \geq 8$ \cite[Corollary~1.6]{m-inv}, thus $u(F) = 8$. We now show that $m(F) > 4$ using essentially the same argument found in the last paragraph of the proof of \cite[Proposition~4.3]{6-dim}.

Let $q$ be any anisotropic quadratic form over $F$ with $\dim q \leq 4$. Then $q$ is already defined over~$E_n$ for some $n$ and $q$ is anisotropic over $E_n$. Next, for any anisotropic form $\rho$ over a field $k$ of characteristic $\ne 2$, one can easily show that the form $\rho_{k(t)} \perp \langle t \rangle$ is anisotropic over $k(t)$ for an indeterminate $t$. So since $q$ is anisotropic over~$E_n$ and $\widetilde{E}_n = E_n(X_n)$ for an indeterminate $X_n$, we conclude that the form $q_{\widetilde{E}_n} \perp \langle X_n \rangle$ is anisotropic over~$\widetilde{E}_n$. Now, as observed by Hoffmann in \cite[Proof of Proposition~4.3]{6-dim}, if $p$ is an anisotropic quadratic form over a field $k$, $\Char k \ne 2$, with $\dim p \leq 5$ and $p$ becomes isotropic over the function field $k(\sigma)$ of a quadratic form $\sigma$ over $k$ with $\dim \sigma \geq 6$, then $p$ is necessarily a five-dimensional Pfister neighbor of some anisotropic 3-fold Pfister form $\varphi$ over $k$. This then implies that $\sigma$ is also a Pfister neighbor of $\varphi$. So, if there is some $m > n$ such that $q_{\widetilde{E}_n} \perp \langle X_n \rangle$ becomes isotropic over $E_m$ (by our definition of $E_m$ and \cite[Theorem~1]{function fields} such an $m$ would need to be even since $\dim(q \perp \langle X_n \rangle) \leq 5$), then the form $\psi'$ is a Pfister neighbor over $E_m$. However, we saw above that $\psi'$ is not a Pfister neighbor over $K$ for any intermediate field $E \subseteq K \subseteq F$. So $q_{\widetilde{E}_n} \perp \langle X_n \rangle$ remains anisotropic over $E_m$ for all $m > n$. Therefore $q \perp \langle X_n \rangle$ is anisotropic over $F$, and since $X_n \in F^{\times}$, we conclude that $q$ is not universal over $F$. Since $q$ was an arbitrary anisotropic quadratic form over $F$ with $\dim q \leq 4$, we conclude that $m(F) > 4$, and the proof is complete.
\end{proof}

\begin{cor}
\label{m_{i,j} separates}
There are fields $k_1, k_2$ of characteristic $\ne 2$ with $u(k_1) = u(k_2)$, $m(k_1) = m(k_2)$, but $m_{i,1}(k_1) \ne m_{i,1}(k_2)$ for $i = 2,3$.
\end{cor}
\begin{proof}
By \cite[Proposition~4.3]{6-dim} there is a linked field $k_1$ with $m(k_1) = 6$. Moreover, Hoffmann's proof shows that $k_1$ can be chosen so that $k_1$ is non-real with $\Char k_1 \ne 2$. Namely, if one takes the base field $E$ in Hoffmann's construction to be non-real of characteristic $\ne 2$ (e.g., $E = \C(x,y,z)$ and $\varphi = \langle \langle x, y \rangle \rangle \perp z \cdot \langle 1, 1 - x \rangle$), then $k_1$ is non-real of characteristic $\ne 2$ as well. By Lemma \ref{linked field with m=6}, we have $u(k_1) = 8$ and $m(k_1) = m_{2,1}(k_1) = m_{3,1}(k_1) = 6$. By Proposition \ref{merkurjev tower}, there is a field $k_2$ of characteristic $\ne 2$ such that $u(k_2) = 8$, $m(k_2) = 6$, and $m_{2,1}(k_2) = 5$. Since $m_{2,1}(k_2) = 5$, we conclude that $m_{3,1}(k_2) \leq 5$ by Lemma \ref{initial inequalities}(a). So for the fields $k_1, k_2$, we have $u(k_1) = u(k_2) = 8$, $m(k_1) = m(k_2) = 6$, but $m_{2,1}(k_1) \ne m_{2,1}(k_2)$ and $m_{3,1}(k_1) \ne m_{3,1}(k_2)$.
\end{proof}

Corollary \ref{m_{i,j} separates} shows that values of $u(k)$ and $m(k)$ for a field $k$ do not determine the values of all the invariants $m_{i,j}(k)$, thereby showing that the invariants $m_{i,j}(k)$ provide information beyond~$u(k)$ and $m(k)$.

\section{Connecting $m_{i,j}$ and $\lgp(r,s)$}
\label{connections}
In \cite{cas}, a ``going-up'' result was proven to hold for the local-global principle for isotropy \cite[Proposition~3.6]{cas}. That is, if quadratic forms over certain fields $k$ of a particular dimension~$n$ satisfy the local-global principle for isotropy with respect to some set $V$ of discrete valuations on~$k$, then all quadratic forms over $k$ of dimension $> n$ also satisfy the local-global principle for isotropy with respect to $V$. Moreover, this dimension $n$ is strongly related to the $u$-invariant of the field~$k$. In this section, we will show that an analogous ``going-down'' result holds for $\lgp(r, s)$ in terms of the refined $m$-invariant (see Theorem~\ref{going-down}).

\begin{lemma}
\label{new ce from old}
Let $k$ be a field of characteristic $\ne 2$ equipped with a non-empty set $V$ of non-trivial discrete valuations. Let $i, j \geq 1$ be positive integers, and let $n$ be a positive integer such that $n < m_{i,j}(k)$. If there is an $n$-dimensional counterexample to $\emph{LGP}(r, j)$ over $k$ with respect to $V$ for some integer $r \geq 1$, then for all integers $s$ such that $1 \leq s \leq i$ there is an $(n+s)$-dimensional counterexample to $\emph{LGP}(r, j)$ over $k$ with respect to $V$.
\end{lemma}
\begin{proof}
Let $q$ be an $n$-dimensional quadratic form over $k$ such that $i_W(q_v) \geq r$ for all $v \in V$ but $i_W(q) < j$. Because $\dim q < m_{i,j}(k)$ and $i_W(q) < j$, there must be an $i$-dimensional form $\sigma_i$ over~$k$ such that $i_W(q \perp \sigma_i) < j$. Let $s$ be any integer such that $1 \leq s \leq i$ and let $\sigma_s$ be any $s$-dimensional subform of $\sigma_i$. Then $i_W(q \perp \sigma_s) < j$, and for all $v \in V$ we have $i_W\left((q \perp \sigma_s)_v\right) \geq i_W(q_v) \geq r$. So $q \perp \sigma_s$ is an $(n+s)$-dimensional counterexample to $\lgp(r, j)$ over $k$ with respect to~$V$.
\end{proof}

\begin{lemma}
\label{less than refined m}
Let $k$ be a field of characteristic $\ne 2$ equipped with a non-empty set $V$ of non-trivial discrete valuations. Let $i, j \geq 1$ be positive integers, and let $n$ be a positive integer such that $n \leq m_{i,j}(k)$. If all $n$-dimensional quadratic forms over $k$ satisfy $\emph{LGP}(r, j)$ with respect to $V$ for some integer $r \geq 1$, then so do all quadratic forms over $k$ of dimension $< n$.
\end{lemma}
\begin{proof}
By contradiction, suppose the lemma is false, and let $n^*$ be the largest integer $< n$ such that there exists an $n^*$-dimensional counterexample to $\lgp(r, j)$ over $k$ with respect to $V$. Our assumptions on $n$ imply $n^* < m_{i,j}(k)$. So by Lemma \ref{new ce from old} there is an $(n^*+1)$-dimensional counterexample to $\lgp(r,j)$ over $k$ with respect to $V$. This is a contradiction of our choice of $n^*$, so the proof is complete.
\end{proof}

\begin{theorem}
\label{going-down}
Let $k$ be a field of characteristic $\ne 2$, let $V$ be a non-empty set of non-trivial discrete valuations on $k$, let $i,j \geq 1$ be positive integers, and let $s$ be an integer such that $1 \leq s \leq i$. If all quadratic forms over $k$ of dimension $m_{i,j}(k) + s - 1$ satisfy $\emph{LGP}(r, j)$ with respect to $V$ for some integer $r \geq 1$, then so do all quadratic forms over $k$ of dimension $< m_{i,j}(k)$.
\end{theorem}
\begin{proof}
By Lemma \ref{less than refined m}, it suffices to show that all quadratic forms over $k$ of dimension $m_{i,j}(k) - 1$ satisfy $\lgp(r, j)$ with respect to $V$. By the contrapositive of Lemma \ref{new ce from old}, because all quadratic forms over $k$ of dimension $m_{i,j}(k) + s - 1$ satisfy $\lgp(r, j)$ with respect to $V$, then so do all quadratic forms over $k$ of dimension $m_{i,j}(k) - 1$, which completes the proof.
\end{proof}

\begin{remark}
There are certain situations in which Theorem \ref{going-down} could be particularly useful when studying $\lgp(1,1)$, i.e., the local-global principle for isotropy. Let $k$ be a field of characteristic~$\ne 2$ with $m(k) = u(k) < \infty$, and let $V$ be a non-empty set of non-trivial discrete valuations on $k$. By definition, any quadratic form over $k$ of dimension $> u(k)$ is isotropic over $k$, and therefore satisfies $\lgp(1,1)$ with respect to $V$. If all quadratic forms over $k$ of dimension~$m(k) = m_{1,1}(k)$ satisfy $\lgp(1,1)$ with respect to $V$, then by Theorem \ref{going-down}, so do all quadratic forms of dimension~$< m(k)$ over $k$. Therefore in order to show that all quadratic forms over $k$ satisfy $\lgp(1,1)$ with respect to~$V$, it suffices to show only that all quadratic forms of dimension $m(k) = u(k)$ do.
\end{remark}

To conclude this article, we ask whether more can be shown about the invariants $m_{i,1}$ than what we learned from Corollary \ref{m_{i,j} separates}. Recall that Corollary \ref{m_{i,j} separates} states that there are fields $k_1, k_2$ with $u(k_1) = u(k_2)$ and $m(k_1) = m(k_2)$, but $m_{i,1}(k_1) \ne m_{i,1}(k_2)$ for $i = 2,3$. So we are interested in the following variant of Question \ref{separate fields?}.
\begin{question}
\label{same m_{i,1}, different m_{n,1}}
Are there fields $k_1, k_2$ with $u(k_1) = u(k_2)$ and $m_{i,1}(k_1) = m_{i,1}(k_2)$ for all integers $1 \leq i < n$, but $m_{n,1}(k_1) \ne m_{n,1}(k_2)$ for some integer $n \geq 3$?
\end{question}
We will now show how the refined local-global principle for isotropy can be used to compute refined $m$-invariants, thereby providing a potential strategy for answering Question \ref{same m_{i,1}, different m_{n,1}}. First, we prove a lemma similar to \cite[Lemma~2.10]{universal}.
\begin{lemma}
    \label{using lgp(r,s) for m_{i,j}}
    Let $k$ be a field of characteristic $\ne 2$ equipped with a non-empty set $V$ of non-trivial discrete valuations. Suppose there exist integers $i,j,n \geq 1$ such that all $(n+i)$-dimensional quadratic forms over $k$ satisfy $\emph{LGP}(j,j)$ with respect to $V$. Let $q$ be an $n$-dimensional quadratic form over $k$ such that $i_W(q_v) \geq j$ for all $v \in V$. Then $i_W(q \perp \sigma_i) \geq j$ for all $i$-dimensional quadratic forms $\sigma_i$ over~$k$. If in addition $i_W(q) < j$, then $m_{i,j}(k) \leq n$.
\end{lemma}
\begin{proof}
    The second statement follows immediately from the first since $q$ is an $(i,j)$-realizing form in this case, so it suffices to prove the first statement. Let $\sigma_i$ be an arbitrary $i$-dimensional form over~$k$ and consider the $(n+i)$-dimensional form $q \perp \sigma_i$ over $k$. Since $i_W(q_v) \geq j$ for all $v \in V$, we have $i_W((q \perp \sigma_i)_v) \geq j$ for all $v \in V$. This, by assumption, implies that $i_W(q \perp \sigma_i) \geq j$ over $k$.
\end{proof}
The next lemma is a generalization of \cite[Lemma~2.9]{universal}.
\begin{lemma}
    \label{globally realizing implies locally realizing}
    Let $k$ be any field of characteristic $\ne 2$, let $v$ be any non-trivial discrete valuation on $k$, and let $i,j \geq 1$ be any positive integers. Suppose $q$ is a quadratic form over $k$ such that $i_W(q \perp \sigma_i) \geq j$ for all $i$-dimensional forms $\sigma_i$ over $k$. Then for all $i$-dimensional forms $\sigma_i^v$ over~$k_v$, $i_W(q \perp \sigma_i^v) \geq j$ over $k_v$.
\end{lemma}
\begin{proof}
    We begin with an observation. Let $a_v \in k_v^{\times}$ be arbitrary. The field $k$ is dense in $k_v$, so by an application of Krasner's Lemma \cite[Proposition~7.61]{milne} we can find an element $a \in k^{\times}$ close enough to $a_v$ in $k_v$ to ensure that $a$ and $a_v$ belong to the same square class of $k_v$. That is, in $k_v$, $a/a_v$ is a square. 

    Now let $\sigma_i^v$ be any $i$-dimensional quadratic form over $k_v$. By the previous paragraph there is an $i$-dimensional quadratic form $\sigma_i$ over $k$ such that $\sigma_i \simeq \sigma_i^v$ over $k_v$. So $q \perp \sigma_i \simeq q \perp \sigma_i^v$ over $k_v$, and since $i_W(q \perp \sigma_i) \geq j$ over $k$, it follows that $i_W(q \perp \sigma_i^v) \geq j$ over $k_v$.
\end{proof}

We now give an example of how Lemmas \ref{using lgp(r,s) for m_{i,j}} and \ref{globally realizing implies locally realizing} can be used to compute the refined $m$-invariants of certain fields. It is not known if fields satisfying the assumptions of the following example exist. However, if they do exist, then they would provide a positive answer to Question~\ref{same m_{i,1}, different m_{n,1}} and show that for each $n \geq 2$ the invariant $m_{n,1}$ provides information beyond the $u$-invariant and the invariants~$m_{i,1}$ for $1 \leq i < n$.
\begin{example}
    \label{fields with different m_{i,j}}
    Let $n$ be any integer such that $n \geq 2$. Let $k_1,k_2$ be fields of characteristic $\ne 2$ such that $u(k_1) = u(k_2) = 2^{n+1}$, $m(k_1) = m(k_2) = 2^n$, and $m_{i,1}(k_1) = m_{i,1}(k_2)$ for all $i$ such that $1 \leq i \leq n-1$. Suppose $k_1$ is equipped with a non-empty set $V_1$ of non-trivial discrete valuations such that $m(k_{1,v_1}) > 2^n$ for all $v_1 \in V_1$ and suppose all $(2^n+1-n)$-dimensional quadratic forms over $k_1$ satisfy $\lgp(1,1)$ with respect to $V_1$. Also, suppose that $k_2$ is equipped with a non-empty set $V_2$ of non-trivial discrete valuations with respect to which there exists a $(2^n+1-n)$-dimensional counterexample to $\lgp(1,1)$ over $k_2$ but all $(2^n + 1)$-dimensional quadratic forms over $k_2$ satisfy $\lgp(1,1)$. Then we claim that $m_{n,1}(k_1) \ne m_{n,1}(k_2)$.

    To prove the claim, we first observe that $2^n+1-n \leq m_{n,1}(k_1), m_{n,1}(k_2) \leq 2^n$ by Lemma~\ref{initial inequalities}(a) and Corollary \ref{decreasing i to 1}. Now consider $k_2$. Then since $(2^n+1)$-dimensional forms over $k_2$ satisfy $\lgp(1,1)$ and there exists a $(2^n+1-n)$-dimensional counterexample to $\lgp(1,1)$ over $k_2$ with respect to~$V_2$, Lemma \ref{using lgp(r,s) for m_{i,j}} implies that $m_{n,1}(k_2) \leq 2^n+1-n$. Hence $m_{n,1}(k_2) = 2^n+1-n$.

    Now consider $k_1$, and by contradiction, assume $m_{n,1}(k_1) = m_{n,1}(k_2) = 2^n+1-n$. This assumption implies that there exists a $(2^n+1-n)$-dimensional $(n,1)$-realizing form $\varphi$ over $k_1$. Since $(2^n+1-n)$-dimensional forms over~$k_1$ satisfy $\lgp(1,1)$ with respect to $V_1$, we conclude that there is some $v_1^* \in V_1$ such that~$\varphi$ is anisotropic over $k_{1,v_1^*}$. Moreover, because $\varphi$ is an $(n,1)$-realizing form over~$k_1$, Lemma~\ref{globally realizing implies locally realizing} implies that~$\varphi$ is an $(n,1)$-realizing form over $k_{1,v_1^*}$, hence $m_{n,1}(k_{1,v_1^*}) \leq 2^n+1-n$. However, we assumed $m(k_{1,v_1^*}) >2^n$, so by Corollary \ref{decreasing i to 1}, $m_{n,1}(k_{1,v_1^*}) > 2^n+1-n$ and we have reached a contradiction. Therefore $m_{n,1}(k_1) \ne m_{n,1}(k_2)$, completing the proof of the claim.
\end{example}

\begin{ack} The author would like to thank David Harbater, Julia Hartmann, Daniel Krashen, and Florian Pop for inspiring discussions and comments regarding the material of this article. This paper is based on part of the author's Ph.D. thesis, completed under the supervision of David Harbater at the University of Pennsylvania.
\end{ack}

\providecommand{\bysame}{\leavevmode\hbox to3em{\hrulefill}\thinspace}
\providecommand{\href}[2]{#2}

\medskip

\noindent{\bf Author Information:}\\

\noindent Connor Cassady\\
Department of Mathematics, The Ohio State University, Columbus, OH 43210-1174, USA\\
email: cassady.82@osu.edu
\end{document}